\documentclass[12pt]{amsart}
\usepackage[margin=1.1in]{geometry}
\usepackage{amsmath,amsthm}
\usepackage{amssymb,amsxtra,verbatim}
\usepackage{latexsym}
\usepackage{color}

\usepackage{graphicx}
\usepackage[version=0.96]{pgf}
\usepackage{tikz}

\usepackage[cmtip,all]{xy}

\usepackage[T1]{fontenc}
\usepackage{mathptmx}
\usepackage[
	hyperindex,
	pagebackref,
	pdftex,
	pdfdisplaydoctitle,
	pdfpagemode=UseNone,
	breaklinks=true,
	extension=pdf,
	bookmarks=false,
	plainpages=false,
	colorlinks,
	linkcolor=blue,
	citecolor=red,
	urlcolor=red
]{hyperref}

\theoremstyle{plain}
\newtheorem{main-theorem}{Main Theorem}

\numberwithin{equation}{section}

\theoremstyle{plain}
\newtheorem{theorem}{Theorem}[section]
\newtheorem{prop}[theorem]{Proposition}
\newtheorem{corollary}[theorem]{Corollary}

\newtheorem{lemma}[theorem]{Lemma}

\theoremstyle{definition}

\newtheorem{definition}[theorem]{Definition}
\newtheorem{example}[theorem]{Example}
\newtheorem{remark}[theorem]{Remark}

\newcommand{\slfrac}[2]{\left.#1\middle/#2\right.}

\def\ra{\rightarrow}

\def\co{\colon\thinspace} 

 \DeclareMathOperator{\Ext}{Ext}

   \DeclareMathOperator{\Def}{Def}

\DeclareMathOperator{\Pic}{Pic}

\DeclareMathOperator{\Tr}{Tr}
\DeclareMathOperator{\Hom}{Hom}
\DeclareMathOperator{\SHom}{\mathit{\mathcal{H}om}}

\DeclareMathOperator{\spec}{Spec}
\DeclareMathOperator{\proj}{Proj}

\DeclareMathOperator{\Spec}{\mathit{Spec}}

\DeclareMathOperator{\Bl}{Bl}
\DeclareMathOperator{\Sym}{Sym}
\DeclareMathOperator{\Hilb}{Hilb}
\DeclareMathOperator{\codim}{codim}

\DeclareMathOperator{\mult}{mult}

\DeclareMathOperator{\Sing}{Sing}
\DeclareMathOperator{\Exc}{Exc}

\DeclareMathOperator{\lct}{lct}
\DeclareMathOperator{\pr}{pr}

\DeclareMathOperator{\q}{\textbf{q}}

\def\Rat{\mathfrak{M}_0}

\def\Hn1{\mathcal{H}_{n,1}}

\renewcommand\S{\mathfrak{S}}
\def\ev{\mathrm{even}}
\def\odd{\mathrm{odd}}
\DeclareMathOperator{\brr}{\mathbf{br}}
\DeclareMathOperator{\HH}{\mathrm{H}}
\DeclareMathOperator{\br}{br}
\def\smooth{\mathrm{smooth}}
\def\irr{\text{irr}}
\def\red{\text{red}}
\def\sh{\text{sh}}

\def\A{\mathcal{A}}
\def\C{\mathcal{C}}
\def\D{\mathcal{D}}

\def\O{\mathcal{O}}

\def\H{\mathcal{H}}
\renewcommand{\P}{\mathcal{P}}

\def\MM{\mathcal{M}}
\newcommand\Mg[1]{\overline{\mathcal{M}}_{#1}}

\newcommand\Mdiv[1]{\mathcal{R}_{#1}}
\newcommand\CMdiv[1]{R_{#1}}

\def\L{\mathcal{L}}
\def\I{\mathcal{I}}
\def\J{\mathcal{J}}
\def\E{\mathcal{E}}
\def\R{\mathcal{R}}
\def\X{\mathcal{X}}
\def\Y{\mathcal{Y}}

\def\T{\mathcal{T}}
\def\W{\mathcal{W}}
\def\m{\mathfrak{m}}

\def\AA{\mathbb{A}}
\def\QQ{\mathbb{Q}}
\def\ZZ{\mathbb{Z}}

\def\PP{\mathbb{P}}
\def\ZZ{\mathbb{Z}}
\def\NN{\mathbb{N}}

\def\GG{\mathbb{G}}

\def\KK{\mathbb{K}}

\newcommand{\M}{\overline{M}}
\newcommand{\Tl}[1]{\mathcal{T}_{#1}}

\def\UU{\mathfrak{U}}

\DeclareMathOperator\Isom{\mathrm{Isom}}
\DeclareMathOperator\PGL{PGL}

\def\Sch{\mathfrak{Sch}}
\def\nb{\nobreakdash}

\begin{document}

\title{Moduli spaces of hyperelliptic curves with A and D singularities}

\author{Maksym Fedorchuk}
\address{Department of Mathematics, Columbia University, 2990 Broadway, New York, NY 10027}
\curraddr{}
\email{mfedorch@math.columbia.edu}
\begin{abstract}
We introduce moduli spaces of quasi-admissible hyperelliptic covers with at worst $A$ and $D$ singularities. Stability conditions for these moduli problems depend on two parameters describing allowable singularities. By varying these parameters, we go from the stacks $\T_{A_n}$ 
and $\T_{D_n}$ of stable limits of $A_n$ and $D_n$ singularities to the quotients 
of their versal deformation spaces by a natural $\GG_m$ action. 
We prove that the intermediate spaces are 
log canonical models of $\T_{A_n}$ and $\T_{D_n}$. 
\end{abstract}

\maketitle 

\setcounter{tocdepth}{1}
\tableofcontents

\section{Introduction}
\label{S:intro}

We begin a systematic study of the interplay between local geometry 
of the versal deformation space of a curve singularity and global geometry of the 
so-called stack of its stable limits. Here, we treat the case of simple planar curve singularities of types $A$ and $D$. The choice is explained by the possibility to treat these singularities in 
a unified fashion using our theory of
quasi-admissible hyperelliptic covers. 

To begin, we associate to a smoothable curve singularity 
the variety of all possible stable limits obtained by applying the stable reduction of 
Deligne and Mumford \cite{DM} to a smoothing of the singularity.
In fact, as follows from the following definition, this variety can be 
given the structure of a Deligne-Mumford stack. 
\begin{definition}[Stack of stable limits]\label{D:variety-stable-limits} 
Let $C$ be a proper integral curve of arithmetic genus $g$ with a single isolated singularity $p$ such that $\hat{\O}_{C,p}$ is smoothable and $\Def(C)$ is irreducible.
Consider a rational {\em moduli map} 
$
j\co \Def(C) \dashrightarrow \Mg{g}
$ and its graph
 \[
\xymatrix{
&Z  \ar[rd]^{q} \ar[ld]_{p} &\\
\Def(C) \ar@{-->}[rr]&&\Mg{g}
}
\]
We define $\Tl{\hat{\O}_{C,p}}:=q(p^{-1}(0))\subset \Mg{g}$ to be the {\em stack of stable limits} of $\hat{\O}_{C,p}$.
\end{definition}
The stack of stable limits was introduced by Hassett in \cite[Section 3]{Hassett-stable}, where the description of $\Tl{\hat{\O}_{C,p}}$ is obtained for
certain toric and quasi-toric planar singularities.

As a variety, $\Tl{\hat{\O}_{C,p}}$ is simply 
the locus of stable curves appearing as stable limits of smoothings of 
$C$. If $b$ is the number of analytic branches of $p\in C$ and $\delta(p)$ is the $\delta$\nb-invariant of $\hat{\O}_{C,p}$,
then curves lying in $\Tl{\hat{\O}_{C,p}}$ have form $\tilde{C} \cup T$, 
where $(\tilde{C}, q_1, \dots, q_b)$ is the pointed normalization of $C$ and 
$(T,p_1,\dots, p_b)$ is a $b$-pointed curve of arithmetic genus 
$\gamma=\delta(p)-b+1$ (here, $(T,p_1,\dots, p_b)$ is attached to $(\tilde{C}, q_1, \dots, q_b)$ by identifying $p_i$ with $q_i$).
The essential information of the stable limit is encoded in 
$(T, p_1,\dots, p_b)$, called the {\em tail of a stable limit of 
$\hat{\O}_{C,p}$}. Tails of stable limits are independent of $\tilde{C}$
and depend only on $\hat{\O}_{C,p}$. 
It follows that $\Tl{\hat{\O}_{C,p}}$ 
is naturally identified with a closed substack of 
$\Mg{\gamma,b}$ (cf. \cite[Proposition 3.2]{Hassett-stable}).

Much attention is devoted in \cite{Hassett-stable} 
to the case of planar $A$ and $D$ singularities.
In particular, it is shown there that a generic tail of a stable limit of 
the $A_n$ singularity ($y^{2}-x^{n+1}=0$) is a smooth
hyperelliptic curve of genus $\lfloor n/2\rfloor$, marked by a Weierstrass point if $n$ is even, or by two 
points conjugate under the hyperelliptic involution if $n$ is odd. 
In the case of the $D_n$ singularity ($x(y^{2}-x^{n-1})=0$), the picture is similar: 
The tails of stable limits are hyperelliptic curves 
of genus $\lfloor (n-1)/2 \rfloor$, marked by 
three points -- two of which are conjugate -- if $n$ is even, or by 
two points -- one of which is a Weierstrass point -- if $n$ is odd. 
This description of $\Tl{A_n}$ and $\Tl{D_n}$ motivated our 
definition of quasi-admissible hyperelliptic covers 
in Section \ref{S:quasi-admissible-covers}.

Understanding varieties of stable limits is crucial
to the study of deformation theory of curve singularities on the one hand, and to the study of birational geometry of 
$\M_g$
on the other hand.
One application of our Main Theorem \ref{main2} is
a functorial {\em simultaneous $(A_k, D_\ell)$-stable reduction} for curves with at worst $A$ and $D$ singularities, generalizing the simultaneous semistable
reduction for ADE curves \cite{radu-yano} in the AD case.
This application is described in more detail in Section \ref{S:final}.

We now discuss the relevance of our results to the Mori-theoretic study 
of $\M_g$ in the program initiated by Hassett and Keel, whose ultimate goal
is the description of log canonical models 
\[
\M_{g}(\alpha):=\proj\bigoplus_{m\geq 0}\HH^{0}(\Mg{g}, \lfloor m(K_{\Mg{g}}+\alpha\delta)\rfloor).
\]
The varieties of stable limits of curve singularities feature prominently in the 
study of $\M_{g}(\alpha)$ because they often lie in the stable base 
loci of $K_{\Mg{g}}+\alpha\delta$, and are, as a rule, 
very special 
loci inside $\M_{g}$. 
Examples include the locus of hyperelliptic curves, 
the locus of trigonal curves of the highest Maroni invariant, 
the Petri divisor in $\M_4$, etc. 
As a result, the varieties of stable limits appear in factorizations 
into blow-ups and blow-downs 
of rational maps between those log canonical models of $\M_g$ that are presently 
understood due to work of Hassett, Hyeon, and Lee \cite{HH1, HH2, HL}.
For example, consider $\M_{g}^{ps}=\M_{g}(9/10)$ -- the coarse moduli space of
at worst cuspidal curves, and $\M_g^{hs}=\M_{g}(7/10-\epsilon)$ -- the coarse moduli 
space of at worst tacnodal curves (for precise definitions see \cite{Schubert} and \cite{HH1, HH2}). 
Then by \cite{HH1}, there is a regular morphism $\M_g\ra \M_g^{ps}$, a divisorial contraction whose exceptional locus is $\Delta_1$ -- the locus of curves with elliptic tails. 
By \cite{HH2},
there is a rational map $\M_g^{ps} \dashrightarrow \M_g^{hs}$, a flip of the locus of curves with {\em elliptic bridges} to the locus of tacnodal curves. 
The relevance of  varieties of stable limits is as follows:
The $1$\nb-dimensional fibers of $\M_{g} \ra \M_{g}^{ps}$ are exactly 
the varieties $\T_{A_2}$ of stable limits of the cusp ($y^2=x^3$); in particular, we 
have $\T_{A_2}\simeq \M_{1,1}$. Further, there exists a resolution of the rational map $\M_g^{ps} 
\dashrightarrow \M_g^{hs}$
 \[
\xymatrix{
W_1\ar[d] \ar[rd] \\
\M_g\ar[d] &W_2 \ar[rd] \ar[ld]&\\
\M_g^{ps} \ar@{-->}[rr]&&\M_{g}^{hs}
}
\]
and the $2$\nb-dimensional fibers of $W_1\ra  \M_{g}^{hs}$ 
are exactly the varieties $\T_{A_3}$ of stable limits of the tacnode $(y^2=x^4)$; in particular, $\T_{A_3}\simeq \M_{1,2}$. 
The $2$-dimensional fibers of
$W_2 \ra \M_{g}^{hs}$ are the loci of curves obtained by varying a genus 
$1$ at worst cuspidal bridge. These loci are isomorphic to 
our $\H_{3}[2]$ of Definition \ref{D:stack-Hnk}.

In a similar fashion, whenever 
a moduli space $\M_{g}[A_k, D_\ell]$ parameterizing proper curves of genus $g$ 
with at 
worst  $A_k$ and $D_\ell$ singularities ($k\geq \ell-1$) is 
constructed,\footnote{Conjecturally, all $\M_{g}(\alpha)$ for $\alpha> 38/71$ are
of this form.} 
the variety $\T_{D_\ell}\subset \M_{g}$ 
will appear inside a strict transform of the $D_\ell$ locus
under the rational map $\M_{g}[A_k, D_\ell]\dashrightarrow \M_g$. 



\subsection{Statement of results}
In this paper, we introduce and study the moduli stacks $\H_{n}[k]$ 
and $\H_{n}[k,\ell]$ of  
{\em quasi-admissible hyperelliptic covers}. We postpone the precise definition to 
Section \ref{S:quasi-admissible-covers} and 
proceed to describe quasi-admissible covers informally. Briefly, the quasi-admissible hyperelliptic covers generalize both the
 {\em admissible covers} of Harris and Mumford \cite{harris-mumford} and the {\em twisted covers} of Abramovich, Corti, and Vistoli \cite{ACV}. 
 Namely,  a quasi-admissible hyperelliptic cover of genus $g$ is a degree $2$ map 
 $
 \varphi\co C\ra R
 $
 such that 
 \begin{enumerate}
 \item $C$ is a curve of arithmetic genus $g$,
 \item $R$ is a tree of pointed rational curves, where
 \item the marked points on $R$ are the branch points of $\varphi$. 
 \end{enumerate}
 By assigning weights to the branch points, we control
the singularities of $C$: Allowing $k$ branch points to collide introduces the 
$A_{k-1}$ singularity ($y^2=x^{k}$) on $C$. By forgetting 
the datum of $C$ and $\varphi$, we obtain a weighted pointed rational curve $R$. We require $R$, marked by the weighted branch divisor, to be stable (see Definition 
\ref{D:stable-div-marked}). 
Finally, in order to have a smooth stack of quasi-admissible covers, 
we never allow ramification over the nodes of $R$. As in the theory of twisted covers
\cite{ACV}, 
this is achieved by introducing an orbicurve structure at certain nodes of $R$ and 
nodes of $C$ lying over them.
The new feature of quasi-admissible covers, as compared to admissible and twisted covers, is that $C$ can have singularities over the smooth locus of $R$. In particular,
on the moduli stack there is a boundary divisor $\delta_{\irr}$ defined as the closure of singular double covers of $\PP^1$ (these arise when two branch points come together) and there is a boundary
divisor $\delta_{\red}$ parameterizing covers with reducible $R$.
\begin{figure}[htb]\label{F:quasi-admissible-cover}
\begin{centering}
\begin{tikzpicture}[scale=0.7]
	\node [style= black] (0) at (0.00, 1.25) {};
	\node [style= black] (1) at (0.00, 1.00) {};
	\node [style= black] (1a) at (-0.25, 1.10) {};
	\node [style= black] (1b) at (0.25, 1.10) {};
	\node [style= black] (2) at (-2.00, 0.50) {};
		\node [style= black] (2) at (-2.25, 0.60) {{\SMALL $A_3: y^2=x^4$}};
	\node [style= black] (3) at (1.00, 0.50) {};
	\node [style= black] (3) at (1.25, 0.75) {};
	\node [style= black] (4) at (2.00, 0.50) {};
	\node [style= black] (5) at (-3.00, 0.25) {};
	\node [style= black] (6) at (-4.00, 0.00) {};
	\node [style= black] (7) at (0.00, 0.00) {};
	\node [style= black] (7a) at (0.25, -0.10) {};
	\node [style= black] (7b) at (-0.25, -.10) {};
	\node [style= black] (8) at (4.00, 0.00) {};
	\node [style= black] (9) at (0.00, -0.25) {};
	\node [style= black] (10) at (2.00, -0.50) {};
	\node [style= black] (11) at (4.00, -0.50) {};
	\node [style= black] (12) at (4.25, -0.50) {};
	\node [style= black] (12) at (4.50, -1) {{\SMALL $A_2: y^2=x^3$}};
	\node [style= black] (13) at (-3.00, -0.75) {};
	\node [style= black] (14) at (-4.00, -1.00) {};
	\node [style= black] (C) at (-4.50, -1.0) [label=above:$C$] {}; 
	\node [style= black] (15) at (4.00, -1.00) {};
	\node [style= black] (16) at (0.00, -2.00) {};
	\node [style= black] (17) at (-1.50, -2.375) [label=below:$4$]{};
	\draw[fill] (17) circle (4pt);
	\node [style= black] (18) at (1.00, -2.25) [label=below:$2$]{};
	\draw[fill] (18) circle (2pt);
	\node [style= black] (19) at (-4.00, -3.00) {}; 
	\node [style= black] (R) at (-4.50, -2.50) [label=below:$R$] {};
	\draw  [very thick, ->]  (C.center) to (R.center);
	\node [style=black] (f) at (-4.25, -1.75) [label=left:$\varphi$] {};
	\node [style= black] (20) at (4.00, -3.00) [label=below:$3$]{};
	\node [style= black] (branch) at (3.00, -2.75) [label=below:$1$] {}; 
	\draw[fill] (branch) circle (1.5pt);
	\draw[fill] (20) circle (3pt);
	\draw [very thick, out=10, in=165] (4.center) to (11.center); 
	\draw [very thick] (1a.center) to (10.center);
	\draw [very thick, out=210, in=0, looseness=1.2] (1b.center) to (6.center);
	\draw [very thick] (19.center) to (16.center);
	\draw [very thick] (7b.center) to (4.center);
	\draw [very thick] (16.center) to (20.center);
	\draw [very thick, out=50, in=150, looseness=0.8] (14.center) to (7a.center); 
	\draw [very thick, out=165, in=-30] (11.center) to (10.center); 
\end{tikzpicture} 
\end{centering}
\caption{A reducible quasi-admissible hyperelliptic 
cover of genus $4$ with $A_{3}$ and 
$A_{4}$ singularities. Numbers indicate 
multiplicities of the branch divisor.}
\end{figure}
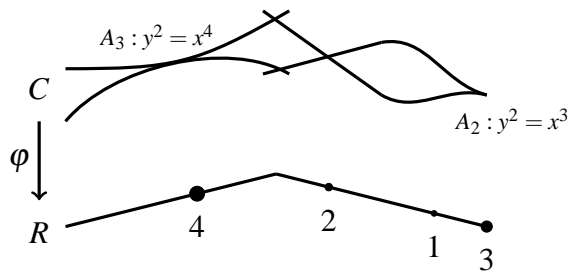

Finally, a few words on how to obtain quasi-admissible covers with $D$ singularities. 
For this, we consider a $1$\nb-pointed variant of quasi-admissible covers: This is done by introducing a marked point $\chi\in C$. 
Whenever $\chi$ coalesces with an $A_{k}$ singularity on $C$,
a $D_{k+1}$ singularity appears. The replacement procedure is described in more detail in Section \ref{S:D-singularities}, where 
the equivalence of deformations of a $D_{k+1}$ singularity and an $A_k$ singularity with a section is established.
For now we only note that our quasi-admissible covers can have at most one $D$ singularity. The reason for this is that a small deformation of a $D$ singularity 
has at most one singularity of type $D$. 
With the replacement procedure of Section \ref{S:D-singularities} in mind, we say that a quasi-admissible cover has a $D_1$ (resp., $D_2$) singularity if 
$\chi$ lies over a branch point of $\varphi$ (resp., a double branch point of $\varphi$).
\begin{main-theorem}[$A_n$ case]\label{main1} Let $n\geq 2$ be an integer.
\begin{enumerate}
\item
For each $k=1,\dots, n-1$,
there exists a smooth and proper Deligne-Mumford stack $\H_n[k]$ representing the functor of quasi-admissible hyperelliptic covers 
with at worst $A_k$ singularities. 
\item There is an isomorphism $\T_{A_n}\simeq \H_n[1]$.
\item There is a sequence of divisorial contractions 
\begin{align*}
\H_n[1]\ra \H_n[2]\ra \dots \ra \H_n[n-1].
\end{align*}
\item There is an isomorphism 
\begin{align*}
\H_n[n-1]\simeq 
\begin{cases} \P(2,3,\dots, n+1),\text{ if $n$ is odd}, \\ 
\P(4,6,\dots, 2n+2),\text{ if $n$ is even.}
\end{cases}
\end{align*}
\item For any $\alpha\in  \left(\frac{1}{2}+\frac{1}{k+2}, \frac{1}{2}+\frac{1}{k+1}\right] \cap \QQ\, $, the coarse moduli space of $\H_{n}[k]$ is
\[
H_{n}(k)=\proj R(\H_n[1], K_{\H_n[1]}+\alpha\delta_\irr+\delta_\red)
\]
where $\delta_{\irr}$, resp. $\delta_{\red}$, 
is the Cartier divisor of irreducible, resp. reducible, singular covers.
\end{enumerate}
\end{main-theorem}
\begin{remark} Note that the threshold value of $\alpha$ at which $$\proj R(\H_n[1], K_{\H_n[1]}+\alpha\delta_\irr+\delta_\red)$$
becomes the coarse moduli space of quasi-admissible covers with at worst $A_k$ singularities is precisely 
\[
\alpha=\lct(A_k)=\frac{1}{2}+\frac{1}{k+1},
\]
where $\lct(A_k)$ is the log canonical threshold of $A_k$ (see Section \ref{S:lct}).
\end{remark}
\begin{main-theorem}[$D_n$ case]\label{main2} Let $n\geq 4$ be an integer.
\begin{enumerate}
\item
For each $1\leq k\leq n-1$ and $1\leq \ell \leq \min\{k+1, n-1\}$,
there exists a smooth and proper Deligne-Mumford stack $\H_n[k,\ell]$ representing the functor of quasi-admissible $1$\nb-pointed
hyperelliptic covers 
with at worst $A_k$ and $D_\ell$ singularities.
\item There is an isomorphism $\T_{D_n}\simeq \H_n[1,2]$.
\item There are divisorial contractions 
{\small
\begin{align*}
\xymatrix{
\H_{n}[1,1] \ar[d]   \ar[r]  & \H_{n}[1,2] \ar[d]\\
\H_{n}[2,1] \ar[r] \ar[d] & \H_{n}[2,2]  \ar[d]  \ar[r] & \H_{n}[2,3] \ar[d]  \\
\cdots \ar[r] \ar[d] & \cdots \ar[r] \ar[d] & \cdots \ar[d]  \ar[r] & \cdots\ar[d]\\
\H_{n}[n-2,1] \ar[r] \ar[d] & \H_{n}[n-2,2] \ar[r] \ar[d] & \cdots \ar[d]\ar[r] & \cdots \ar[r] \ar[d]& \H_{n}[n-2,n-1] \ar[d] & \\
\H_{n}[n-1,1] \ar[r]  & \H_{n}[n-1,2] \ar[r] & \cdots \ar[r] & \H_{n}[n-1,n-2] \ar[r] 
& \H_{n}[n-1,n-1]  &
}
\end{align*}
}
\item There is an isomorphism 
\[
\H_{n}[n-1,n-1] =\begin{cases} \P(\frac{n}{2}, 1,2,3,\dots,n-1), &\text{ if $n$ is even,}\\
 \P(n, 2,4,6,\dots,2n-2), &\text{ if $n$ is odd}.
 \end{cases}
\]
\item For any $\alpha\in  \left(\frac{1}{k+2}, \frac{1}{k+1}\right] \cap \QQ$ and
$\beta\in \left(1-(\ell+1)\alpha, 1-\ell \alpha\right]\cap \QQ$, 
\[
H_{n}[k,\ell]=\proj R(\H_n[1,1], K_{\H_n[1,1]}+(\alpha+1/2)\delta_\irr+(2\alpha+2\beta-1)\delta_W+\delta_\red),
\]
where $\delta_{\irr}$, resp. $\delta_{\red}$, is the Cartier divisor
of irreducible, resp. reducible, singular covers, and $\delta_{W}$ is 
the {\em Weierstrass
divisor} of covers with a marked ramification point.
\end{enumerate}
\end{main-theorem}

We note that $A$ and $D$ singularities come equipped with a $\GG_m$ 
action and so do their versal
deformation spaces (see Pinkham \cite{pinkham} for a systematic theory
of singularities with $\GG_{m}$ action).
Parts (3) of Main Theorems \ref{main1} and \ref{main2} 
give a description of $\T_{A_n}$ and $\T_{D_n}$ 
as an iterated weighted blow-up of weighted projective spaces and 
Parts (4) identify these weighted projective spaces with 
quotient stacks $\left[\Def(A_n)\smallsetminus \mathbf{0} \ /\  \GG_m\right]$ 
and $\left[\Def(D_n)\smallsetminus \mathbf{0} \ /\  \GG_m\right]$.
Parts (1) provide a functorial interpretation of the intermediate spaces as moduli 
stacks of quasi-admissible hyperelliptic covers with 
$A$ and $D$ singularities. 


\subsubsection*{Roadmap of the proof:} The moduli stacks
$\H_n[k]$ and $\H_{n}[k,\ell]$ of Parts (1) of Main Theorems are
defined in Section \ref{S:quasi-admissible-covers}.
They are shown to be smooth and proper Deligne-Mumford stacks
in Theorems \ref{T:DM} and \ref{T:H-smooth-stack}. The 
morphisms of Parts (3) of Main Theorems are
constructed in Section \ref{S:reduction}. Parts (5) of Main Theorems are
dealt with in Section \ref{S:logMMP} (Theorems \ref{T:logMMP-A} and
\ref{T:logMMP-D}). The proofs of Parts (4)  are contained in Example 
\ref{E:weighted-projective-spaces}. We note that the moduli 
space $\H_{n}[n-1]$ is also studied in \cite{JFD}, where another proof of Part (4) of Main Theorem \ref{main1} can be found. The proof there is deduced from a
global quotient representation of
$\H_{n}[n-1]$, valid in arbitrary characteristic, given in \cite[Theorem 4.1]{arsie-vistoli}. 
Finally, Parts (2) of Main Theorems follow from definitions.

\subsection{Notation and conventions}\label{S:notations}
\subsubsection{The base field.} We work throughout over an algebraically closed field 
$\mathbb{K}$ of characteristic $0$.
The characteristic $0$ assumption is used in an essential way at several points of this work. 
In particular, the stacks $\H_n[k]$ and $\H_n[k,\ell]$ of Main Theorems 
\ref{main1} and \ref{main2}  
are smooth Deligne-Mumford stacks only under the characteristic $0$ assumption,
as Example \ref{E:infinite-automorphisms} shows.
However, for a fixed integer $n$, the statements of Main Theorems
\ref{main1} and \ref{main2} remain valid in characteristic $p$ as long as $p> n+1$.
\subsubsection{Curves and their singularities}\label{S:curves-definitions}
Unless specified otherwise, a {\em curve} 
is a connected reduced finite type scheme of dimension one over an algebraically closed field. 
Recall that a curve (singularity) $C$ 
is {\em smoothable} if there is a flat family $f\co \C\ra T$
with $f^{-1}(0)\simeq C$ and $f^{-1}(t)$ smooth for $t\neq 0$. 
All planar curve singularities are smoothable and moreover 
have nonsingular deformation
spaces -- these are the only curve singularities encountered in this paper. 

A singularity of {\em type $A_{n}$} is analytically isomorphic to 
$y^{2}-x^{n+1}=0$ at $(0,0)$. A singularity of {\em type $D_{n}$} 
is analytically isomorphic to 
$x(y^{2}-x^{n-2})=0$ at $(0,0)$. We note that $A_3\simeq D_3$.

By a theorem of Arnold (see \cite{arnold, arnold-inventiones, AGLV}), 
the $A$ and $D$
singularities, together with three exceptional singularity 
types $E_{6}, E_{7}$ and $E_{8}$, are the only 
{\em simple} hypersurface singularities (in every dimension). 
A singularity is called simple if it admits no nontrivial equisingular deformation, 
or, equivalently, if it has no {\em moduli}. 
From the point of view of moduli theory of curves, these are the simplest classes of singularities 
to allow with a hope of obtaining an open substack of the stack of all reduced curves. 

When $C$ is a proper curve, we use $\Def(C)$ to denote the (uni)versal deformation space of $C$, and $\Delta_{C}$ to denote the {\em discriminant} -- the locus 
of singular deformations -- inside $\Def(C)$. When $p\in C$ is a
singular point, we use $\Def(\hat{\O}_{C,p})$ to denote the versal deformation space of the singularity.
It is a standard fact (see, e.g., \cite{tjurina-versal} or \cite[Chapter 3.1]{sernesi}) that the versal deformation space of an isolated 
hypersurface singularity defined by $f(x_1,\dots,x_n)=0$ around $(0,\dots, 0)$ is the 
finite-dimensional $\mathbb{K}$\nb-vector space 
underlying the {\em Tjurina algebra}
\[
\mathbb{K}[x_1,x_2,\dots,x_n]/\left(f, \partial f/\partial x_1, \dots, \partial f/\partial x_n\right).
\]
We record versal deformations of singularities of types $A$ and $D$ and
the corresponding $\GG_{m}$ actions. 
We have $\Def(A_n)=\spec \KK[a_{0},\dots, a_{n-1}]$ and the 
versal deformation is given by
\begin{equation*}\label{E:A-versal}
y^2-(x^{n+1}+a_{n-1}x^{n-1}+\dots+a_0)=0.
\end{equation*}
The $\GG_m$ action on $\Def(A_n)$ and the versal family over it, extending 
a natural $\GG_m$ action on $y^2-x^{n+1}=0$, is given by 
\begin{equation}\label{E:A-action}
\begin{aligned}
\lambda\cdot (x,y, a_{n-1}, \dots, a_0) &=(\lambda^{2}x, \lambda^{n+1}y, \lambda^{4}a_{n-1},\dots, \lambda^{2(n+1)}a_0) \quad \text{if $n$ is even,} \\
\lambda\cdot (x,y, a_{n-1}, \dots, a_0) &=(\lambda x, \lambda^{\frac{n+1}{2}}y, \lambda^{2}a_{n-1},\dots, \lambda^{n+1}a_0) \quad \text{if $n$ is odd.}
\end{aligned}
\end{equation}
We have $\Def(D_n)=\spec \KK[b, a_{0},\dots, a_{n-2}]$ and the versal deformation is given by
\begin{equation*}
xy^2+by-(x^{n-1}+a_{n-2}x^{n-2}+\dots+a_0)=0.
\end{equation*}
The $\GG_m$ action on $\Def(D_n)$ and the versal family over it, extending 
a natural $\GG_m$ action on $x(y^2-x^{n-2})=0$, is given by
\begin{equation}\label{E:D-action}
\begin{aligned}
\lambda\cdot (x,y, b, a_{n-1}, \dots, a_0) &=(\lambda x, \lambda^{\frac{n-2}{2}}y, \lambda^{\frac{n}{2}} b, \lambda a_{n-2},\dots, \lambda^{n-1}a_0) \quad \text{if $n$ is even,} \\
\lambda\cdot (x,y, b, a_{n-1}, \dots, a_0) &=(\lambda^2 x, \lambda^{n-2}y, \lambda^{n} b, \lambda^2 a_{n-2},\dots, \lambda^{2n-2}a_0) \quad \text{if $n$ is odd.}
\end{aligned}
\end{equation}

\subsubsection{Log canonical thresholds}\label{S:lct}
A  {\em log canonical threshold} of a hypersurface $D\subset \mathbb{K}^n$ quantifies how
far the hypersurface is from being a simple normal crossing divisor. We refer the reader to
\cite[Section 8]{singularities-pairs} for precise definition. Here, we only record the log canonical thresholds of $A_{n}$ and $D_{n}$ singularities:
\begin{align*}
\lct(A_{n}) &=\frac{n+3}{2(n+1)}=\frac{1}{2}+\frac{1}{n+1}, \\ 
\lct(D_{n}) &=\frac{n}{2(n-1)}=\frac{1}{2}+\frac{1}{2(n-1)}.
\end{align*}
An amusing fact is that the above thresholds are related to the geometry of the discriminant 
hypersurface inside the versal deformation space of the curve singularity.
Namely, let $p\in C\subset \mathbb{K}^{2}$ be a singularity of type $A$ or $D$. Let $\Delta$ be the discriminant hypersurface 
inside the versal deformation space $\Def(\hat{\O}_{C,p})$. We then have
\begin{align*}
\lct(\Delta, \Def(\hat{\O}_{C,p}))=\lct(C, \mathbb{K}^2).
\end{align*}

\subsubsection{Conventions} An {\em Artin stack} is a stack that has a separated 
representable diagonal of finite type 
(i.e., $\Isom$ functors are represented by separated algebraic spaces of finite type) 
and that admits a representable smooth surjective morphism from 
a scheme; a {\em Deligne-Mumford stack} is an Artin stack
with an unramified diagonal (i.e., objects have no infinitesimal automorphisms). 
 If $\X$ is a Deligne-Mumford stack,
we denote by $X$ its coarse moduli space. 
An {\em orbicurve} is a Deligne-Mumford stack of finite type over an algebraically 
closed field
whose coarse moduli space is a curve, and such that the generic stabilizer
of every irreducible 
component is trivial. 
Unless specified otherwise, a {\em family} 
is a flat family of schemes or orbicurves.


\subsubsection{Notation.} The symmetric group on $d$ letters is denoted by 
$\S_{d}$; the cyclic 
group of order $r$ by $\mu_r$. If $D$ is a $\QQ$\nb-Cartier divisor on 
$X$, we set 
$$R(X,D):=\bigoplus_{m\geq 0} \HH^0(X, \lfloor mD\rfloor).$$
The category of schemes over 
$\spec A$ is denoted $\Sch_A$. When 
$\GG_m$ acts on $\mathbb{K}^{n+1}$ diagonally with weights $a_0,\dots,a_n$,
we denote by 
$\P(a_0,\dots,a_n)$ the {\em weighted projective stack} 
$
\left[ \mathbb{K}^{n+1} \smallsetminus \mathbf{0}\ / \ \GG_m \right]
$. The coarse moduli space of $\P(a_0,\dots,a_n)$ is
$\PP(a_0,\dots,a_n)$.


\subsection{Outline of the paper} In Section \ref{S:moduli-spaces}, we discuss {\em divisorially marked rational curves}, analogues of weighted pointed curves of
\cite{Hweights}, and the notion of $\W$\nb-stability for them, an analogue of $\A$\nb-stability. The second part of the section deals with {\em even rational 
orbicurves} -- divisorially marked rational curves endowed with the minimum stack structure allowing for 
existence of a square root of the marking divisor. In Section \ref{S:quasi-admissible-covers}, we introduce (pointed) quasi-admissible 
hyperelliptic covers with at worst $A$ (and $D$) singularities 
and discuss the notion of $\W$\nb-stability for them. Here, we prove that stacks $\H_{n}[k]$ and $\H_{n}[k,\ell]$ of (pointed) quasi-admissible covers are smooth and proper Deligne-Mumford stacks, and show that their local geometry closely reflects the geometry of versal deformation spaces of $A$ and $D$ singularities. 
Section \ref{S:logMMP} is devoted to the study of certain log canonical divisors on 
$\H_{n}[k]$ and $\H_{n}[k,\ell]$. In particular, we show that the natural reduction 
morphisms constructed in Section \ref{S:reduction} are, on the level of coarse 
moduli spaces, maps between log canonical models 
of $\Tl{A_{n}}$ and $\Tl{D_{n}}$. Finally, in Section \ref{S:final}, we discuss the necessity of working with Deligne-Mumford stacks of 
orbicurves and give an application of Main Theorems \ref{main1} and \ref{main2}
related to a recent work of \cite{radu-yano}.


\subsection*{Acknowledgements} We would like to thank Jarod Alper, Brendan Hassett, Aise Johan de Jong, and David Smyth for helpful discussions. We are also grateful to the organizers of several seminars in which this work was presented before the manuscript was finished. We also thank Sebastian Casalaina-Martin and Radu Laza who shared the preliminary version of \cite{radu-yano} and encouraged the inclusion of Section \ref{S:final}.

\section{Divisorially marked rational curves and even rational orbicurves}
\label{S:moduli-spaces}

\subsection{Divisorially marked rational curves}
We begin by summarizing the theory of pointed and divisorially marked rational curves.
Throughout, a {\em semistable rational curve} will be a
proper, connected, (at worst) nodal curve of arithmetic genus $0$.
A {\em divisorially marked rational curve} 
is a semistable rational curve $C$ together with divisors $D_i$ of degree $d_i$ that are disjoint from $\Sing(C)$.
Given two $(n+1)$-tuples, $(d_0, d_1,\dots, d_n)\in \NN^{n+1}$ and $(w_0, w_1, \dots, w_n)\in \QQ^{n+1}\cap(0,1]^{n+1}$, we call 
the datum $\W:=(w_0^{d_0},w_1^{d_1}, \dots,w_n^{d_n})$ a {\em weight vector} 
(we also write $w_{i}^{1}$ as $w_{i}$). 
\begin{definition}\label{D:stable-div-marked} Given a weight vector $\W$, we say that a divisorially marked rational curve $(C; D_0,\dots, D_n)$ is {\em $\W$\nb-stable} if the following conditions hold:
\begin{enumerate}
\item For every $p\in C$, $\mult_p \sum\limits_{i=0}^n w_i D_i\leq 1$.
\item The line bundle $\omega_C\bigl(\sum\limits_{i=0}^n w_i D_i\bigr)$ is ample.
\end{enumerate}
\end{definition}
Consider the {\em stack of divisorially marked curves} $\Mdiv{} \ra \Sch_{\mathbb{K}}$ whose objects are families 
$\pi\co (\C; D_0, D_1, \dots, D_n) \ra T$
such that the geometric fibers of $\pi$ are divisorially marked rational curves. 
To see that $\Mdiv{}$ is a stack, note that it is a category fibered in grupoids, that 
the $\Isom$ functors are
sheaves in the \'{e}tale topology, 
and that the \'{e}tale descent for objects is effective.

\begin{definition}\label{D:W-stable} Let $\Mdiv{\W}$ be the full subcategory of $\Mdiv{}$ 
consisting of families
whose geometric fibers are $\W$\nb-stable divisorially marked rational curves. 
We call $\Mdiv{\W}$ the {\em stack of rational $\W$\nb-stable curves.}
\end{definition}
The stack $\Mdiv{\W}$ is closely related to the moduli space
 of weighted pointed rational curves introduced by Hassett in \cite{Hweights}. 
We recall that for $\A=(a_1,\dots, a_n)\in (0,1]^n\cap \QQ^n$, 
a proper, connected, (at worst) nodal rational curve $(C; p_1,\dots, p_n)$ is $\A$-stable if it satisfies:
\begin{enumerate}
\item $p_i$ are smooth points of $C$, and if $p_{i_1}, \ldots, p_{i_k}$ coincide in $C$, then $\sum_{j=1}^{k}a_{i_j} \leq 1$.
\item $\omega_{C}(\sum\limits_{i=1}^n a_ip_i)$ is ample.
\end{enumerate}
The moduli functor of $\A$\nb-stable curves is represented by a smooth projective scheme $\Mg{0,\A}$ \cite[Theorem 2.1]{Hweights}. 
The relation between $\Mdiv{\W}$ and $\Mg{0,\A}$ is summarized in the following proposition.
\begin{prop}\label{P:DivStack-X} The category $\Mdiv{\W}$ is a smooth and proper Deligne-Mumford stack over $\mathbb{K}$. 
Its coarse moduli space $\CMdiv{\W}$ is
a quotient of $\Mg{0,\A}$, 
where {\small $\A=\bigl(\underbrace{w_0,\dots,w_0}_{d_0},\dots,\underbrace{w_n,\dots,w_n}_{d_n}\bigr)$}, by the action of $\S:=\S_{d_0}\times \cdots \times\S_{d_n}$. 
\end{prop}
\begin{proof}
Consider the forgetful $1$\nb-morphism $\Mdiv{\W}\ra \Rat$ to the category 
of semistable rational curves with at most $1+\sum_{i=0}^{n}d_i$ nodes. 
Note that $\Rat$ is an Artin stack of finite type over $\mathbb{K}$
 by, e.g., \cite[Proposition 1.10]{Fulg1}. Take $T\ra \Rat$ to be a smooth surjective 
morphism from a separated scheme $T$ of finite type over $\mathbb{K}$. 
This morphism induces a family $\C_T\ra T$ 
of semistable rational curves.
The fiber product $\Mdiv{\W}\times_{\Rat}T$ is the category whose objects over a $T$-scheme $S$ 
are $(n+1)$-tuples of $S$\nb-flat Cartier divisors 
$\D_0,\dots, \D_n$ on $\C_T\times_T S$. 
The morphisms in $\Mdiv{\W}\times_{\Rat}T$ are obvious Cartesian diagrams 
and the isomorphisms in $\Mdiv{\W}\times_{\Rat}T$
are equalities on the nose.
It follows that $\Mdiv{\W}\times_{\Rat}T$ is an open subscheme of $\Hilb_{\C_T/T}$ -- 
the relative Hilbert scheme of $\C_T\ra T$. Since $\C_T$
is a separated algebraic space of finite type over $T$, $\Hilb_{\C_T/T}$ 
is also an algebraic space of finite type over $T$ by \cite{rydh}, 
and so admits an \'{e}tale surjective morphism from a finite type scheme. 
Composing this morphism 
with a smooth surjective morphism $\Mdiv{\W}\times_{\Rat}T \ra \Mdiv{\W}$, we obtain 
a smooth surjective cover of $\Mdiv{\W}$ by a scheme of finite type over $\mathbb{K}$.
Thus, $\Mdiv{\W}$ is an Artin stack of finite type over $\mathbb{K}$.

To show that $\Mdiv{\W}$ is a smooth Deligne-Mumford stack it remains to show that for any $\W$\nb-stable curve defined 
over a field, infinitesimal automorphisms and obstructions vanish.
This follows from the well-known deformation-theoretic result of Proposition \ref{P:deformation-divisors} below.

Since $\Mdiv{\W}$ is a Deligne-Mumford stack of finite type over $\mathbb{K}$, 
the coarse moduli space exists by \cite[Corollary 1.3]{keel-mori} and the coarse moduli map 
is proper \cite[Theorem 3.1(1)]{conrad}. 
It is easy to see that the coarse moduli space is in fact isomorphic 
to the scheme-theoretic quotient $\Mg{0,\A}/ \S$, which is a proper
scheme. 
Thus, $\Mdiv{\W}$ is proper.
\end{proof}
\begin{prop}\label{P:deformation-divisors}
Let $(C; D_{0},\dots, D_{n})$ be a $\W$-stable rational curve over $\mathbb{K}$. 
Then $(C; D_{0},\dots, D_{n})$ is unobstructed and has no infinitesimal automorphisms.
\end{prop}
\begin{proof}
The first statement follows from the fact that $C$ is unobstructed and $D_i$ are 
Cartier divisors supported at smooth points of $C$. To prove the second statement, we note
that the infinitesimal automorphisms are elements $d\in \Hom(\Omega^{1}_{C}, \O_C)$ that 
satisfy $df\in (f)$ for any local equation $f$ of $D_0+\cdots+D_n$. 
In characteristic $0$, this shows that infinitesimal automorphisms are classified by 
$\Hom(\Omega^{1}_{C}, \O_C(-D))\simeq \HH^0(C, (\omega_{C}(D))^{-1})=(0)$. Here,
the first isomorphism and the vanishing statement is taken from \cite[Section 3.3]{Hweights}.
\end{proof}
We remark that by \cite[Section 3.3.2]{Hweights} the infinitesimal automorphisms of 
$(C; D_0,\dots, D_n)$ are classified by $\Ext^{1}(\Omega_{C}^1\langle D_0,\dots, D_n\rangle, \O_C)$,
where $\Omega_{C}^1\langle D_0,\dots, D_n\rangle$ is the sheaf of differentials on $C$ with 
logarithmic poles along $D_0,\dots, D_n$.
\begin{example}\label{E:infinite-automorphisms}
We remark that even a smooth divisorially marked rational curve $C$ 
over a field $\mathbb{K}$ of positive 
characteristic $p$ may have infinitesimal automorphisms. Indeed, if $\text{char}\ \mathbb{K} =p$, then
by the proof of Proposition \ref{P:deformation-divisors} above, the infinitesimal automorphism 
$x\mapsto x+\epsilon y$ of $\proj \mathbb{K}[x,y]$ extends to an automorphism of a closed immersion
given by \[
\mathbb{K}[x,y] \ra \mathbb{K}[x,y]/\left((x-a_1y)^{pd_1}(x-a_2y)^{pd_2}\dots(x-a_ny)^{pd_n}\right).
\]
\end{example}

\subsubsection{Divisorially marked vs. pointed curves}
\label{S:different-marked-curves}
Let $\A=\bigl(\underbrace{w_0,\dots,w_0}_{d_0},\dots,\underbrace{w_n,\dots,w_n}_{d_n}\bigr)$ and $\W=\bigl(w_0^{d_0},\dots, w_n^{d_n}\bigr)$.
Set $\S:=\S_{d_0}\times \cdots \times\S_{d_n}$. We have a sequence of $1$\nb-morphisms
\[ 
\Mg{0,\A} \ra \left[\Mg{0,\A} / \S\right] \ra \Mdiv{\W} \ra \CMdiv{\W}\simeq \Mg{0,\A}/\S,
\]
defined as follows: The leftmost arrow is the quotient map and the rightmost arrow is the coarse moduli map. To define the middle arrow, we note that 
an object of $\left[\Mg{0,\A} / \S\right]$ over a scheme $T$ is an $\S$-torsor $\{P\ra T\}$ 
together with an $\S$-equivariant morphism $P \ra \Mg{0,\A}$. The $1$-morphism 
$\left[\Mg{0,\A} / \S\right] \ra \Mdiv{\W}$ sends $\{P\ra T\}$ to the family 
of divisorially marked curves over $T$
\[
\left(P\times_{\Mg{0,\A}}\C_{\A}\right) /\S \ra P/\S=T,
\]
where the divisor $D_i$ on $\left(P\times_{\Mg{0,\A}}\C_{\A}\right) /\S$ is defined to be the image of 
weight $w_i$ sections of $P\times_{\Mg{0,\A}}\C_{\A}\ra P$.
 Note that the middle arrow $\left[\Mg{0,\A} / \S\right] \ra \Mdiv{\W}$ 
is not representable since $\Mdiv{\W}$ usually has smaller stabilizers, 
as we see in the following example.
\begin{example} Let $\A=\bigl(1, 1/3, 1/3, 1/3, 1/3\bigr)$ and $\W=\bigl(1, (1/3)^4\bigr)$. 
To describe $\Mg{0,\A}$, note that any $\A$\nb-stable curve has only one irreducible component, 
namely $\PP^1$. Next, we can assume that
the point of weight $1$ is always at $\infty$, and $4$ points of weight $1/3$ have coordinates
$x_1,x_2,x_3,x_4$ 
satisfying $x_1+x_2+x_3+x_4=0$ on the affine line 
$\spec S:=\spec \mathbb{K}[x_1,x_2,x_3,x_4]= \PP^1\smallsetminus \infty$. 
The stability assumption then translates into 
$(x_1, x_2, x_3, x_4)\neq (0,0,0,0)$. 
Since the subgroup of $\PGL_2$ preserving above choices is $\GG_m$, 
we see that $\Mg{0,\A}$ is the quotient stack
\[ 
\left[\spec S /(x_1+x_2+x_3+x_4)\smallsetminus \mathbf{0} \ / \ \GG_m\right] 
\simeq \proj  S /(x_1+x_2+x_3+x_4) \simeq \PP^2;
\]
here, the action of $\GG_m$ is given by the usual grading of $S$.
By the same logic, $\Mdiv{\W}$ can be identified with the quotient stack
\[
\left[\spec (S/(x_1+x_2+x_3+x_4))^{\S_4} \smallsetminus \mathbf{0} \ / \ \GG_m\right],
\] where 
$ (S/(x_1+x_2+x_3+x_4))^{\S_4}$ is the ring of invariants of $S/(x_1+x_2+x_3+x_4)$ under the action of $\S_4$. 
The generators of $(S/(x_1+x_2+x_3+x_4))^{\S_4}$ are elementary symmetric polynomials of degrees $2$, $3$ and $4$. It follows that $\Mdiv{\W}$ is the weighted projective stack 
$\P(2,3,4)$. Finally, $\left[\Mg{0,\A} / \S_4\right]$ is 
$\left[ \proj  S /(x_1+x_2+x_3+x_4) \ / \ \S_4\right]$, where $\S_4$ acts 
by permuting variables. Note that the point with $x_1=x_2=x_3=-x_4/3$
has stabilizer $\S_3$ in $\left[\Mg{0,\A} / \S\right]$ and maps to a point with 
a cyclic stabilizer $\mu_3$ in $\Mdiv{\W}$.
\end{example}

\subsubsection{Section at infinity} In all of the cases under consideration in this paper, 
the weight vector is of the form $\left(1, w_1^{d_1},\dots, 
w_n^{d_n}\right)$. In other words, all families of divisorially marked curves will always
carry a distinguished divisor 
of relative degree $1$ and weight $1$. Such a divisor defines a section, called 
{\em the section at infinity}. 

\subsubsection{Odd nodes and odd section at infinity} If $(C; D_0, D_1,\dots, D_n)$ 
is a divisorially marked rational curve, where $D_0$ is the section at 
infinity, we say that a node of $C$ is {\em odd} (resp., {\em even}) 
if it separates $C$ into two connected 
components $C_1$ and $C_2$ such that
$D_0$ lies on $C_1$ and such that the degree of the divisor $\sum_{i=1}^n D_i$ restricted to
 $C_2$ is odd (resp., even). Moreover, we say that the section at infinity 
is {\em odd} (resp., {\em even}) if the total degree $d_1+\cdots+d_n$ of $\sum_{i=1}^n D_i$ is
 odd (resp., even). The odd nodes and the odd section at infinity are referred to as {\em odd 
 points} of $C$. 

\subsection{Even rational orbicurves}
The definitions in this section are inspired by the notion of a twisted cover of \cite{ACV} (see also \cite{AV}). We use the gadget of orbicurves, to borrow a metaphor from \cite{ACV}, 
as a magnifying glass in which apparent singularities of moduli spaces disappear.

\begin{definition} An {\em even rational orbicurve} is a triple $(\Y; \tau, D)$, where 
\begin{enumerate}
\item $\Y$ is an orbicurve over an algebraically closed field $k$ with a coarse moduli space $Y$.
\item $Y$ is a semistable rational curve.
\item $\tau\co \spec k \ra \Y$ is a section, called {\em the section at infinity}, such that $\tau(\spec k)$ is a smooth point of $Y$.
\item $D$ is a divisor in the smooth locus of $Y$, disjoint from $\tau(\spec k)$. 
\item The points of $\Y$ with non-trivial stabilizers lie exactly over odd points of $Y$.
\item \'{E}tale locally over an odd node of $Y$, the orbicurve $\Y$ is isomorphic to $$\bigl[ \spec \slfrac{k[x,y]}{(xy)} \, / \mu_2 \bigr]$$
where the action of $\mu_2$ is given by $(x,y)\mapsto (-x,-y)$.
\item \'{E}tale locally over an odd section at infinity, the orbicurve $\Y$ is isomorphic to $$\bigl[ \spec k[x] \, / \mu_2 \bigr]$$
where the action of $\mu_2$ is given by $x\mapsto -x$.
\end{enumerate}
\end{definition}
\noindent
The points of $\Y$ with a non-trivial stabilizer ($\mu_{2}$) 
are called {\em odd points}, or {\em orbinodes}.
\begin{lemma} 
\label{L:root-bundle-unique}
Suppose that $(\Y; \tau, D)$ is an even rational orbicurve over an algebraically closed field. 
Then there is a unique $\L\in \Pic(\Y)$ satisfying $\L^{\otimes 2}\simeq \O_\Y(D)$.
Moreover, a character of $\mu_2$ at a point $p\in \Y$ is non-trivial if and only if $p$ is an odd 
point of $\Y$.
\end{lemma}
\begin{proof} 
Recall that a line bundle $\L$ on $\Y$ is a datum of a line bundle $L$ on $Y$ together with an action of 
$\mu_2$ on the fiber $\L_p$ for every odd point $p$ of $Y$.  In particular, $\L^{\otimes 2}$ is always a pullback of a line bundle from  
$Y$. Hence, the uniqueness holds because $\Pic(Y)$ is torsion-free and $\Pic(Y)\ra \Pic(\Y)$ is injective. It remains to establish existence.


We proceed by induction on the number of odd nodes. Suppose there are none. 
Then either $\Y$ is a scheme and the 
statement clearly holds, or $\Y$ has an odd section at infinity. In the latter case, 
let $\Y_1$ be the irreducible component of $\Y$ containing
the section at infinity $\tau$. Then $\Y=\Y_1\cup \Y_2$, a nodal union of two connected components. 
Moreover, $\Y_2$ is a scheme and
the degree of $D$ restricted to every component of $\Y_2$ is even. It follows that there is a line bundle 
$\L_2$ such that $(\L_2)^{\otimes 2}\simeq\O_{\Y_2}(D\vert_{\Y_2})$.
Let $D_1:=D\vert_{\Y_1}+p_1$, where $p_1$ is an arbitrary smooth point of $\Y_1$. 
Since $D_1$ has even degree, there is a line bundle $\L_1$ such that
$(\L_1)^{\otimes 2}\simeq\O_{\Y_1}(D_1)$. Finally, let $\L\in \Pic(\Y)$ be such that 
$\L\vert_{\Y_1}\simeq \L_1(-\frac{1}{2}\tau)$ and $\L_{\vert \Y_2}\simeq \L_2$. Clearly,
$\L^{\otimes 2}\simeq \O_{\Y}(D)$.

Suppose now $\Y$ has an odd node $p$. Let $\Y_1$ and $\Y_2$ be the connected components 
of $\Y$ such that $\Y_1\cap \Y_2=p$ and 
$\Y=\Y_1\cup \Y_2$. Take $R_1$ and $R_2$ to be irreducible components of $\Y_1$ and $\Y_2$, respectively, containing point $p$.
Let $D_i:=D\vert_{\Y_i}+p_i$, where $p_i$ is an arbitrary smooth point of $R_i$. Denote by 
$\tilde{\Y}_i$ the orbicurve obtained from $\Y_i$ 
by forgetting the stack structure at $p$. 
Then $(\tilde{\Y}_i, D_i)$ are even rational orbicurves whose odd nodes are
exactly the odd nodes of $(\Y,D)$, with the exception of $p$. By induction, there are line bundles 
$\L_i$ on $\tilde{\Y}_{i}$ such that $(\L_i)^{\otimes 2}=\O_{\tilde{\Y}_i}(D_i)$. 
Now, form a line bundle $\L$ on $\Y$ satisfying $\L\vert_{\Y_i}=\L_i(-\frac{1}{2}p)$. 
Then $(\L)^{\otimes 2}=\O_\Y(D)$.

\end{proof}

Next, we analyze families of even rational orbicurves over more general bases. 
We begin by fixing a weight vector
$\W=\left(1, w_1^{d_1},\dots, w_n^{d_n}\right)$. As the following lemma illustrates, 
the presence of the section at infinity greatly simplifies the geometry of an arbitrary 
divisorially marked family.
\begin{lemma}
\label{L:zero-section} Let $P \ra T$ be a $\PP^1$-bundle with a section $\tau \co T\ra P$. 
Suppose that there is a $T$-flat divisor $D\subset P$ of relative degree $d$ and
disjoint from $\tau(T)$. Then there exists a section $\sigma\co T\ra  P\smallsetminus \tau(T)$.
\end{lemma}
\begin{proof}
The idea of the proof is to take the center of mass of the divisor in each fiber 
(this is where division by $d$ comes in and the characteristic $0$ assumption is used). 
We now formalize this idea. To begin, we show that the section exists affine locally on $T$. 
To this end, suppose that $T$ is affine and $P=\PP \E$,
where $\E$ is a free vector bundle of rank $2$. Then $\tau(T)$ is the vanishing locus of some
$x\in \HH^0(T, \E^*)$, and the relative divisor $D$ is defined as the vanishing locus of 
$f_D\in \HH^0(T, \Sym^d\E^*)=\Sym^d \HH^0(T, \E^*)$. Since $\E$ is free of rank $2$, we can find 
another section of $P\ra T$ disjoint from $\tau(T)$. 
It corresponds to $y\in \HH^0(T, \E^*)$. We now express $f_D$ in terms of $x$ and $y$:
\[
f_D=a_d x^d+\cdots+a_1xy^{d-1}+a_0y^{d}.
\]
Note that the assumption that $D$ is disjoint from $\tau(T)$ implies that $a_0\neq 0$. We now define $s:=y+\dfrac{a_1}{da_0}x\in \HH^0(T,\E^*)$. The vanishing locus of $s$ defines another section of $P\ra T$, disjoint from $\tau(T)$.

It remains to show that different sections $s$ glue. For this, we need to show that the construction 
of $s$ above was in fact independent of the choice of $y$. Indeed, suppose
we chose section $y'=ay+bx$, with $a\neq 0$. Then using $y=(y'-bx)/a$, we rewrite $f_D$ in the new coordinates as
\begin{multline*}
f_D=a_d x^d+\cdots+a_1x((y'-bx)/a)^{d-1}+a_0((y'-bx)/a)^{d} \\ =\frac{1}{a^d} (a_da^dx^d+\cdots+(aa_1-a_0db)x(y')^{d-1}+a_0(y')^{d}).
\end{multline*}
We see that 
\[s'=y'+\frac{aa_1-a_0db}{da_0}x=ay+bx+\frac{aa_1}{da_0}x-bx=as.\]
 Since $a\neq 0$, $s'$ and $s$ define the same section.
\end{proof}

The principal application of Lemma \ref{L:zero-section} is to the study of the Picard group of 
a $\PP^1$-bundle over a general base. 
Namely, suppose $D$ is a divisor on the $\PP^1$\nb-bundle 
$P\ra T$ of relative degree $d$, and disjoint from the section at infinity $\tau(T)$.
Then by the lemma, there exists a section $\sigma\co T\ra P\smallsetminus \tau(T)$.
Denote the image of $\sigma$ by $\Sigma$.
The divisor $D-d\Sigma$ is of relative degree $0$. Since 
$\HH^1(\PP^1,\O_{\PP^1})=0$, 
the cohomology and base change theorem implies that 
$\O_P(D-d\Sigma)$ is a pullback of a line  
bundle from the base. By the construction, $\O_P(D-d\Sigma)\vert_{\tau(T)} \simeq \O_T$. It follows that $\O_P(D)\simeq \O_P(d\Sigma)$. 

\begin{lemma} 
 \label{L:even-line-bundle}
Let $\pi\co Y \ra T$ be a family of semistable rational
curves with the section at infinity $\tau\co T\ra Y$. 
Suppose that a line bundle $\L\in \Pic(Y)$ has even degree when restricted to every irreducible component of every fiber and satisfies $\tau^{*}\L\simeq \O_T$. Then there is a unique line bundle 
$\MM\in \Pic(Y)$ satisfying $\L\simeq \MM^{\otimes 2}$ and $\tau^{*}\MM\simeq \O_T$.
\end{lemma}
\begin{proof}
Denote the image of the section at infinity by $\Theta$.
The square root of $\L$ restricted to every fiber 
is unique (by, e.g., Lemma \ref{L:root-bundle-unique}). Since the fibers are projective, connected, 
and have $\HH^1(X,\O_X)=0$, 
by the cohomology and base change theorem
there can be at most one line bundle $\MM$ such that $\MM^{\otimes 2}=\L$ and 
$\MM\vert_{\Theta} \simeq \O_\Theta$.

We now prove existence. For every $t\in T$ consider the base extension 
$\O_{T,t}^\sh \ra T$ and the pullback family
$Y'\ra \spec \O_{T,t}^\sh$. On the fiber of $Y'$ over a closed point $t$,
choose smooth points $p_1,\dots, p_d \in Y'_{t}\smallsetminus \Theta$ such that 
$\L_{t}=\O_{Y'_t}(\sum\limits_{i=1}^d\varepsilon_i p_i)^{\otimes 2}$,
where $\varepsilon_i$
are appropriate signs.
By smoothness, the points $p_i$ give rise to
sections $\sigma_i\co \spec \O_{T,t}^\sh \ra Y'$ with images 
$\Sigma_i\subset Y'$ \cite[2.2, Prop. 5]{neron-models}.
By degree consideration and triviality of $\Pic(\O_{T,t}^\sh)$, we have 
$\L\otimes \O_{Y'}\simeq\O_{Y'}(\sum\limits_{i=1}^d\varepsilon_i \Sigma_i)^{\otimes 2}$.
Moreover, 
$\O_{Y'}(\sum\limits_{i=1}^d\varepsilon_i \Sigma_i)$  
restricts to a trivial line bundle along $\Theta$.

It follows that there is a surjective \'{e}tale cover $T'\ra T$ such that the pullback 
of $\L$ to $Y':=Y\times_{T}T'$ has a square root $\MM$ that restricts to a trivial
 line bundle along $\Theta$. (By abuse of notation, we denote by
 $\Theta$ the preimage of $\Theta$ under any base extension.) 
 Fix an isomorphism 
 $\iota \co \MM_{\vert \Theta} \ra \O_{\Theta}$. By the uniqueness above, 
 the line bundles $\pr_1^*\MM$ and $\pr_2^*\MM$ are
isomorphic on $Y'\times_Y Y'$. Fix an isomorphism $\alpha\co 
\pr_1^*\MM\simeq \pr_2^*\MM$ such that the following diagram
commutes: 
\[
\xymatrix{
(\pr_1^*\MM)\vert_{\Theta} \ar[r]^{\alpha\vert \Theta} \ar[d]^{\pr_1^*\iota} & 
(\pr_2^*\MM)\vert_{\Theta} \ar[d]^{\pr_2^*\iota} \\
\pr_1^*\O_{\Theta} \ar[r]^{\textrm{can}} & \pr_2^*\O_{\Theta}
}
\]
It follows that the cocycle condition 
$\pr_{12}^*(\alpha)\circ\pr_{23}^*(\alpha)=\pr_{13}^*(\alpha)$ 
on $Y'\times_{Y} Y'\times_{Y} Y'$ is
satisfied. Thus, $\MM$ descends to a line bundle on $Y$ whose square is $\L$.
\end{proof}

\subsection{Moduli stack of even rational orbicurves} 
Fix a weight 
vector $\W=(1, w_1^{d_1},\dots, w_n^{d_n})$. A family of {\em even rational orbicurves}
over a scheme $T$ will be a datum $(\pi \co \Y\ra T; D_0, D_1,\dots, D_n)$ where $\pi\co \Y \ra T$ 
is a flat and proper morphism from a Deligne-Mumford stack to a scheme such that 
geometric fibers of $\pi$ are even rational orbicurves. The divisor $D_0$ 
of weight $1$ defines the section at infinity $\tau\co T\ra \Y$. From now on, we will conflate $D_0$
and $\tau$. A family of even rational orbicurves 
is called $\W$-stable if the coarse moduli space $Y\ra T$ is 
$\W$-stable in the sense of Definition \ref{D:stable-div-marked}.

\begin{definition}\label{D:even-orbicurves}
We denote by $\Mdiv{\W}^\ev$ the category
of $\W$-stable families of even rational orbicurves. The morphisms in $\Mdiv{\W}^\ev$ 
are obvious cartesian diagrams.
\end{definition}
In the remainder of this section, we show that 
the category $\Mdiv{\W}^\ev$ is a smooth and proper Deligne-Mumford stack over $\mathbb{K}$.

To begin, we recall that given a triple $(X,D,r)$ consisting of 
an arbitrary Deligne-Mumford stack $X$, a Cartier divisor $D$, and a positive integer $r$,
there exists a {\em root stack} $X_{D,r}$: The objects in $X_{D,r}$ over a scheme $T$ 
 are morphisms $f\co T\ra X$ together with a datum of a triple $(\L, s, \iota)$ consisting
 of a line bundle $\L\in \Pic(T)$, a section $s\in \HH^{0}(T,\L)$ and an isomorphism
 $\iota\co \L^r \ra f^*\O_X(D)$ satisfying $\iota(s^r)=f^*(D)$. 
 We refer to \cite{cadman} for the
 construction of $X_{D,r}$ and the proof that $X_{D,r}$ is Deligne-Mumford.

Next, given a Cartier divisor $D$ with irreducible
components $D_1,\dots, D_k$, we define (with a certain abuse of notation)
 \[
 X\sqrt[r]{D} := X_{D_1,r}\times_X \cdots \times_X X_{D_k,r}.
 \]
\begin{lemma}\label{L:root-stack}
If $X$ is a proper Deligne-Mumford stack, then so is $X\sqrt[r]{D}$. Moreover, if
$X$ is smooth and $D$ is a simple normal crossing divisor with each $D_i$ smooth, 
then $X\sqrt[r]{D}$ is also smooth.
\end{lemma}
\begin{proof}
 Briefly, \'{e}tale locally on $X$, the morphism $X\sqrt[r]{D}\ra X$ is given by 
 \[
 \left[ \spec A[x_1,\dots,x_k]/(x_1^r-f_1,\dots, x_k^r-f_k) \ / \ (\mu_r)^k \right] \ra \spec A,
 \]
 where $f_i\in A$ is a local equation of $D_i$. 
Our assumptions imply that, by the Jacobian criterion,
$\spec A[x_1,\dots,x_k]/(x_1^r-f_1,\dots, x_k^r-f_k)$ 
 is smooth. Thus
$X\sqrt[r]{D}$ is smooth.

To check properness, we use the valuative criterion \cite[Proposition 7.12]{LM}. Let $R$ be a discrete
valuation ring with the uniformizer $t$
and the fraction field $K$. Consider a morphism $\spec K \ra X_{D_{i},r}$. Compose 
with $X_{D_{i},r}\ra X$ and use properness of $X$ to conclude that, possibly after a finite base change, 
there is a unique
extension $\phi\co \spec R\ra X$. It remains to note that since $R$ is a unique 
factorization domain, we have $\Pic(\spec R)=0$ and the
$r^{\text{th}}$ root of $\phi^*(D_{i})\in R$ exists after further base change $t=s^r$, and 
is unique up to a unit in $R[s]$. Thus, we obtain $\spec R[s]\ra X_{D_{i},r}$.

To see that $X\sqrt[r]{D}$ is proper we could also observe that, \'{e}tale locally,
$X\sqrt[r]{D}\ra X$ is a map from a stack to its coarse moduli space. In particular,
it is proper \cite[Theorem 3.1(1)]{conrad}.
\end{proof}

 \begin{theorem} 
\label{T:even-orbicurves}
Let $\W=\left(1, w_1^{d_1},\dots, 
w_n^{d_n}\right)$ be a weight vector with $\sum_{i=1}^n d_i$ even. Then 
$\Mdiv{\W}^\ev$ is the 
root stack $\Mdiv{\W}\sqrt[2]{\delta_\odd}\, $, where $\delta_{\odd}$ is the Cartier divisor of $\W$\nb-stable rational curves with odd nodes.
In particular, $\Mdiv{\W}^\ev$ is a smooth and proper Deligne-Mumford stack over $\mathbb{K}$.
\end{theorem}
\begin{proof} 
Recall that $\Mdiv{\W}$ is a smooth and proper Deligne-Mumford stack by Proposition 
\ref{P:DivStack-X}. It carries  
a simple normal crossing Cartier divisor $\delta_{\odd}$ parameterizing curves with odd nodes. 
We proceed to establish the equivalence of categories $\Mdiv{\W}^\ev$
and $\Mdiv{\W}\sqrt[2]{\delta_\odd}$. 

Consider a family of even rational orbicurves $(\pi \co \Y\ra T; \tau, D_1,\dots, D_n)$. 
The coarse moduli space $Y$ is a flat and proper family of divisorially marked rational
curves over $T$. Therefore, it induces a morphism $T\ra \Mdiv{\W}$.
Now suppose $t\in T$ is a point such that the fiber $Y_t$ has 
$k$ odd nodes $p_1,\dots, p_k$. Denote by $\delta_i$ the locus where the node $p_i$ is preserved; $\delta_i$ is an irreducible 
component of $\delta_\odd$. Note 
that, since components of $\delta_{\odd}$ do not self-intersect, the Cartier 
divisors $\delta_i$ 
are distinct. 
By definition, the local equation of $\Y\ra T$ around $p_i$
is 
$$\bigl[ \spec \O_{T,t}^{\sh}\{x,y\}/(xy-t^{a_i}_i)\ /\ \mu_2 \bigr] \ra \spec \O_{T,t}^{\sh},$$
where the action is $(x,y) \mapsto (-x,-y)$, and where $t_i\in \O_{T,t}^\sh$ is (a pullback
of) a local equation of $\delta_i$. Since 
$$\left(\O_{T,t}^{\sh}\{x,y\}/(xy-t^{a_i}_i)\right)^{\mu_2}=
\O_{T,t}^{\sh}\{x^2,y^2,xy\}/(xy-t^{a_i}_i)=\O_{T,t}^{\sh}\{x^2,y^2\}/(x^2y^2-t^{2a_i}_i),$$
 the local equation of $Y\ra T$ around $p_i$ is 
$$\spec \O_{T,t}^{\sh}\{x^2,y^2\}/(x^2y^2-t^{2a_i}_i) \ra \spec \O_{T,t}^{\sh}.$$
Therefore, the morphism $T\ra \Mdiv{\W}$ factors \'{e}tale locally through the root stack
$\left(\Mdiv{\W}\right)_{\delta_i,2}$ for each $i=1,\dots, k$. Thus it
factors \'{e}tale locally through $\Mdiv{\W}\sqrt[2]{\delta_\odd}$. Since the \'{e}tale descent for objects in $\Mdiv{\W}\sqrt[2]{\delta_i}$ is effective, we obtain 
a morphism $T\ra \Mdiv{\W}\sqrt[2]{\delta_\odd}$. 
Since morphisms in both categories are given by fiber products and since the
formation of coarse moduli space commutes with an arbitrary base change in 
characteristic $0$ (e.g., by \cite[Lemma 2.3.3]{AV}), we obtain a natural
transformation $\Mdiv{\W}^{\ev}\ra \Mdiv{\W}\sqrt[2]{\delta_\odd}$.

In the other direction, consider a smooth surjective 
morphism from a scheme $T$ to $\Mdiv{\W}\sqrt[2]{\delta_\odd}$.
Since $\Mdiv{\W}\sqrt[2]{\delta_\odd}$ is smooth and irreducible, $T$ is smooth and 
can be chosen to be irreducible. The map $T\ra \Mdiv{\W}\sqrt[2]{\delta_\odd}$ 
defines 
 a family $(Y\ra T; \tau, D_1,\dots, D_n)$ of $\W$-stable rational curves together with 
 a datum 
 $\{(\L_i, t_i): t_i\in \HH^0(T,\L_i)\}$ for every irreducible 
 component $\delta_i$ of $\delta_\odd$. The local equation of $Y$ around a node 
 $p_{i}\in Y_t$
 corresponding to $\delta_i$ is 
 \[
 \spec \O_{T,t}^{\sh}\{x,y\}/(xy-s_i^{a_i}),
 \]
 and $s_i=t_i^2$ under some identification of $(\L_i)_t$ and $\O_{T,t}$. 
 Next, we consider the blow-up along the ideal $\left((x,y)^{2a_i},t_i\right)$ (it is a weighted blow-up with weights of $x,y,t$ being $1,1,2a_i$) and 
 denote by $E_i$ the exceptional divisor of the blow-up. 
The result is a semistable family $Y' \ra T$ of rational nodal curves. 
 In the fiber $Y_t$ an odd node $p_i$ of $Y_t$
 has been replaced by a rational curve $(E_i)_t$. Evidently, 
 the degree of $\O_{Y'}(E_i)$ restricted to $(E_i)_t$ is even. 
 It follows by Lemma
 \ref{L:even-line-bundle} that the divisor 
 $$
 B:=\sum_{i=1}^n D_i+\sum_{i: \delta_i\subset \delta_{\odd}} E_i
 $$ is divisible by $2$ in the Picard group of 
$Y'$. Let $X'\ra Y'$ be the $\mu_2$-cover totally branched
 over  $B$. Then $X'\ra T$ is again a family of semistable curves, and we denote by $X\ra T$ 
 its stabilization. 
 The action of $\mu_2$ descends to $X$ and
 the stack $\Y=[X/\mu_2]$ is a family of even rational 
 orbicurves over $T$. Clearly, the coarse moduli space of $\Y$ is $Y$. We conclude that there
 is a morphism $T\ra \Mdiv{\W}^{\ev}$. Moreover, since our construction commutes with a
 smooth base change, this morphism descends to give a natural transformation 
 $\Mdiv{\W}\sqrt[2]{\delta_\odd}\ra \Mdiv{\W}^{\ev}$.

Evidently, the two constructed natural transformations define an equivalence of categories.
\end{proof}

\begin{corollary}\label{C:even-orbicurves-DM} The category $\Mdiv{\W}^{\ev}$ of even rational orbicurves
is a Deligne-Mumford stack over $\mathbb{K}$.
\end{corollary}
\begin{proof}
In the case $\sum_{i=1}^n d_i$ is even, this is the content of Theorem \ref{T:even-orbicurves}
above. The case of odd $\sum_{i=1}^n d_i$ reduces to it by a simple trick: Take $\W'=(1,1^3,w_1^{d_1},\dots, w_n^{d_n})$
and consider the morphism $\Mdiv{\W}^{\ev} \ra \Mdiv{\W'}^{\ev}$ that takes a family of even $\W$\nb-stable rational orbicurves and attaches a fixed $4$\nb-pointed $\PP^1$ to the section at infinity to 
obtain a $\W'$-stable curve. Clearly, this morphism is a closed immersion and we are done.
\end{proof}

\section{Quasi-admissible hyperelliptic covers}\label{S:quasi-admissible-covers}

In this section, we weld the notion of an admissible cover \cite{harris-mumford}
and that of a twisted cover \cite{ACV} in the special case of degree $2$. 
We call the result a {\em quasi-admissible hyperelliptic cover}. 
In words, a quasi-admissible  
cover is a finite locally free cover of an even rational orbicurve. 
A hyperelliptic cover differs from a twisted cover in 
that it is allowed to have non-trivial ramification over a smooth locus and it differs from an admissible cover in that branch 
points can come together. In particular, it 
is clear what happens when several branch points collide: a singularity of type $A$
appears on the cover. 

In what follows, we define the moduli stack of (pointed) quasi-admissible covers described in the previous paragraph. To motivate this definition,
we list the desired specifications of this stack: 
The stack has to be proper, singularities of its objects need to be controllable,
the deformation theory has to be tractable. 
Our moduli stack of quasi-admissible covers meets 
all of these requirements: 
By assigning weights to branch points, we can specify how many can collide
and thus specify the allowable singularities of the cover. 
When more than allowed branch points come together, we let 
the cover and the target to sprout out additional components. This makes
our moduli stack proper. Finally, the deformations of all objects are unobstructed, thus the moduli stack is actually smooth. 

\begin{definition}[Quasi-admissible covers]
\label{D:stack-H} 
Define $\H \ra \Sch_{\mathbb{K}}$ to be the stack whose objects over a scheme $T$ are the diagrams 
\begin{equation}\label{E:stack-H}
\xymatrix{
\X 
\ar[r]^{\varphi}  & \Y \ar[d]_{\pi} & D \ar@{_{(}->}[l] \\ 
& T\ar@/_1pc/[u]_{\tau} \ar@/^1pc/[ul]^{(\tau_1, \tau_2)} & }
\end{equation}
satisfying the following properties: 
\begin{enumerate} 
\item
$\pi\co (\Y; \tau, D) \ra T$ is a proper flat family of even rational orbicurves.

\item
$\varphi$ is a finite locally free morphism of degree $2$, 
branched exactly over {\em the branch divisor} $D$
and \'{e}tale elsewhere. 

\item
If $\tau$ is odd, that is, locally \'{e}tale on $T$ we have
$\tau\co \spec A\ra \left[ \spec A[x]\  / \mu_2\right]$, 
then $\varphi$ looks like an \'{e}tale cover 
\[
\left[\spec A[x,t]/(t^2-u) \ / \mu_2\right]\ra \left[ \spec A[x]\  / \mu_2\right],
\]
where $u\in A^{\times}$ and the action is $(x,t)\ra (-x,-t)$.

\item\label{C:sections-tau} 
If $\tau$ is even, then there are sections
 $\tau_1,\tau_2\co T\ra \X$ satisfying 
 $\varphi^*(\tau)=\tau_1+\tau_2$.

\item \label{C:nodes-local} 
Over an odd node $\left[ \spec A[x,y]/(xy) \ / \mu_2 \right]$ of $\Y$ the morphism $\varphi$ looks like
an \'etale cover
\[
\left[\spec A[x,y,t]/(xy, t^2-u) \ /\mu_2\right] \ra\left[\spec A[x,y]/(xy) \ / \mu_2 \right],
\] where $u\in A^{\times}$ and the action is $(x,y,t)\ra (-x,-y,-t)$.

\end{enumerate}
The shorthand notation for an object in $\H$ is $(\varphi\co \X \ra \Y; \tau, D)$.
We call $\H$ the {\em stack of quasi-admissible covers}.
\end{definition}
If $\varphi\co \X \ra \Y$ is a quasi-admissible cover over a field $\mathbb{K}$, 
then passing to the morphism
between coarse moduli spaces $\phi\co X \ra Y$, we see that over the 
odd section at infinity
$\phi$ looks like $\spec \mathbb{K}[x^2,xt,t^2]/(t^2-u) \ra \spec \mathbb{K}[x^2]$,
 which becomes 
$\spec \mathbb{K}[X,Y^2-uX]\ra \spec \mathbb{K}[X]$
after the substitution $X=x^2, Y=xt$. Thus, $\phi$ is ramified over the odd section at 
infinity. When this happens, we denote $\tau_1=\phi^{-1}(\tau)$. Evidently, 
$\tau_1$ defines a section of $X\ra T$.  


By realizing $\X$ as a locally principal
subscheme of an $\AA^1$\nb-bundle over $\Y$, we see that the only 
singularities of $\X$ lying
over smooth points of $\Y$ are of type $A$ ($y^2=x^{k+1}, \ k\geq 1$).

\begin{definition}[Pointed quasi-admissible covers]
\label{D:stack-H1} 
Let $\H^\chi \ra \Sch_{\mathbb{K}}$ be the stack whose objects over a scheme $T$ are diagrams
\begin{equation}\label{E:stack-H1}
\xymatrix{
\X 
\ar[r]^{\varphi}  & \Y \ar[d]_{\pi} & D \ar@{_{(}->}[l]  \\ 
& T\ar@/_1pc/[u]_{\tau} \ar@/_0.5pc/[ul]^{\chi} \ar@/^1pc/[ul]^{(\tau_1, \tau_2)} }
\end{equation}
such that the diagram obtained by forgetting $\chi$ satisfies Definition \ref{D:stack-H}. 
The shorthand notation for an object in $\H^{\chi}$ is $(\varphi\co \X \ra \Y; \tau, \chi, D)$.
We call $\H^\chi$ the {\em stack of pointed quasi-admissible covers}.
\end{definition}

\begin{lemma}\label{L:splitting} In Definition \ref{D:stack-H}, we have 
$
\X\simeq\Spec_{\Y} (\O_\Y\oplus \L^{-1}),
$
where $\L$ is 
a line bundle satisfying $\L^{2}=\O_\Y(D)$ and 
$\tau^*\L\simeq \O_T$.
The $\O_\Y$-algebra structure on $\O_\Y\oplus \L^{-1}$ 
is given by $\L^{-2} \stackrel{\cdot D}{\ra}\O_\Y$. 

\end{lemma}
\begin{proof}
By the characteristic $0$ assumption,
the trace morphism $\Tr\co \varphi_*\O_\X \ra \O_\Y$
followed by scaling by $1/2$ gives a splitting 
$\varphi_*\O_X\simeq \O_\Y\oplus \L^{-1}$, for 
some line bundle $\L$ on $\O_\Y$. Evidently, $\L^2\simeq \O_\Y(D)$. 
Since $D$ is disjoint from $\tau$, by twisting by $\pi^{*}(\tau^{*}\L^{-1})$ we
arrange for $\tau^*\L\simeq \O_T$
without violating the condition 
that $\L^2\simeq \O_\Y(D)$.
That $\X\simeq\Spec_{\Y} (\O_\Y\oplus \L^{-1})$ follows from the fact that
$\varphi$ is a finite morphism. 
\end{proof}

\begin{definition}[Branch morphisms] \label{D:branch-morphism}
The {\em branch morphism} $\brr\co \H \ra \Mdiv{}$ is a forgetful functor  
sending the family in Display \eqref{E:stack-H} to the family of divisorially marked rational curves 
$(Y; \tau, D)$.
The {\em branch morphism}
$\brr\co \H^\chi \ra \Mdiv{}$ is a forgetful functor  sending the family in Display \eqref{E:stack-H1} to the family of 
divisorially marked rational curves $(Y; \tau, \varphi(\chi(T)), D)$.
\end{definition}


Recall that $\Mdiv{\W}$ is the moduli stack of $\W$\nb-stable 
divisorially marked rational 
curves (cf. Section \ref{S:moduli-spaces}).
Below we use the branch morphism to define {\em stability conditions} for
quasi-admissible covers with $A$ and $D$ singularities.

\begin{definition}[Stable quasi-admissible covers with $A$ singularities]\label{D:stack-Hnk}

Fix an integer $n\geq 2$ and a rational number $\alpha\leq 1/2$.
Consider the 
weight vector
$\W=\left(1, \alpha^{n+1}\right)$. 

We define $\H_{n,\alpha}$ to be the moduli stack of quasi-admissible covers
 $(\varphi\co \X\ra \Y; \tau, D)$
such that $(\Y; \tau, D)$ is a $\W$-stable curves. Alternatively, 
$\H_{n,\alpha}=\H\times_{\Mdiv{}} \Mdiv{\W}$, where $\H$ maps to
 $\Mdiv{}$ via the branch morphism of Definition \ref{D:branch-morphism}.
The objects of $\H_{n, \alpha}$ are called {\em $\W$-stable quasi-admissible covers}.

If $1/(k+2)<\alpha\leq 1/(k+1)$ for some $k\in \{1,\dots, n-1\}$, then 
we denote $\H_{n,\alpha}$ by $\H[k]$. For such $\alpha$ 
the geometric points of $\H_{n,\alpha}$ are quasi-admissible covers $(\varphi\co \X\ra \Y;\tau, D)$ such that
\begin{enumerate}
\item the branch divisor $D$ is of degree $n+1$ and of weight $\alpha$,
\item the section at infinity $\tau$ is of weight $1$,
\item $(Y;\tau, D)$ is $\W$\nb-stable rational curve,
\item $\X$ has 
at worst $A_{k}$ singularities,
\item The arithmetic genus of $X$ is $\left\lfloor \dfrac{n}{2}\right\rfloor$.
\end{enumerate}
\end{definition}

\begin{definition}[Stable quasi-admissible covers with $D$ singularities]
Fix an integer $n\geq 4$ and rational numbers
$\alpha\in (0,1/2]$ and $\beta\in (0,1-\alpha]$. 
Consider the weight vector $\W=(1, \beta, \alpha^n)$. 
We define the stack $\H_{n,\alpha,\beta}$ to be the fiber product 
$\H^\chi\times_{\Mdiv{}} \Mdiv{\W}$, 
where $\H^\chi$ maps to $\Mdiv{}$ via the branch morphism of Definition \ref{D:branch-morphism}.
The objects of $\H_{n, \alpha, \beta}$ are called {\em $\W$-stable pointed quasi-admissible covers}.

We now explicate the meaning of $\W$\nb-stability for pointed 
quasi-admissible covers.
To this end, consider the unique integers $k$ and $\ell$ such that $\alpha$ 
and $\beta$ satisfy inequalities
\begin{eqnarray}
\label{inequalities}
\frac{1}{k+2} < \alpha \leq \frac{1}{k+1}, \\ 
1-(\ell+1)\alpha< \beta \leq 1-\ell \alpha. \notag
\end{eqnarray}
If $(\varphi\co \X\ra \Y; \tau, \chi, D)$ is a geometric point of $\H_{n,\alpha,\beta}$, then
\begin{enumerate}
\item $\X$ has at worst $A_{k}$ singularities, since at most $k+1$ branch points can coalesce,
\item the marked point $\chi$ can coalesce with (at worst)
 $A_{\ell-1}$ singularity on $\X$,
 \item The arithmetic genus of $X$ is $\left\lfloor (n-1)/2 \right \rfloor$.
\end{enumerate}
Because of the above, we denote $\H_{n,\alpha,\beta}$ by $\H[k,\ell]$ when
emphasizing the 
singularities allowed on $X$. 
\end{definition}
\begin{example}[{A case study of $\H_{n}[n-1]$ and $\H_{n}[n-1,n-1]$}]
\label{E:weighted-projective-spaces} We begin our study of moduli stacks 
of quasi-admissible covers by proving Parts (4) of 
Main Theorems \ref{main1} and \ref{main2}. 
First, we deal with the case of $\H_{n}[n-1]$,
the moduli stack of $\W$-stable quasi-admissible covers, where $\W=(1, (1/n)^{n+1})$. 
The stability assumption on  $(\varphi\co \X\ra \Y; \tau, D)$ implies that 
$\Y$ is necessarily $\PP^1$ and
the support of the branch divisor $D$ has at least $2$ distinct points, i.e.,
 not all $n+1$ points
can come together.
We denote by $\infty$ the section at infinity. Since it has weight $1$, no branch point can lie at $\infty$.
Further analysis depends on the parity of $n$. 

We first treat the case 
of $n$ odd.
In this case, a hyperelliptic cover is uniquely determined by a branch divisor. It remains to describe 
the moduli space of degree $n+1$ divisors on $\AA^{1}$ not supported at a single point. In characteristic
$0$ every such divisor can be brought into the {\em normal form} 
\[
x^{n+1}+a_{n-1}x^{n-1}+\dots+a_{0}=0,
\] 
with not all $a_i$ zero.
A subgroup of automorphisms of $\PP^{1}$ fixing $\infty$ and 
preserving the normal form is $\GG_{m}$. It 
acts on 
$
\spec \mathbb{K}[a_{n-1},\dots, a_{0}] \smallsetminus  \mathbf{0}
$
diagonally with weights $(2,\dots, n+1)$. It follows that our moduli space is the weighted projective stack 
$\P(2,3,\dots, n+1)$.

A reader might recall that $\P(2,3,\dots, n+1)$ is the stack whose objects over a scheme $T$ are 
data $(\L, s_2, \dots, s_{n+1})$ of a line bundle $\L\in \Pic(T)$ and sections 
$s_m\in \HH^0(T,\L^{m})$. 
We briefly sketch the proof of the equivalence between two categories.

Given a $\W$-stable cover 
$\left(\varphi\co \X \ra \Y; \tau, D \right)$, where $\W=(1,(1/n)^{n+1})$,
all geometric fibers of $\pi\co \Y\ra T$ 
are isomorphic to $\PP^1$. The existence of $\tau$ implies that $\pi \co \Y\ra T$ is a $\PP^1$\nb-bundle.
Since $D\subset \Y$ is a $T$\nb-flat divisor on $\Y$ disjoint from $\tau(T)$, 
by Lemma \ref{L:zero-section}, there exists a section $\sigma\co T\ra \Y$, also disjoint from $\tau(T)$. Consider the short exact sequence
\begin{align}\label{E:bundle-ses}
0 \ra \O_\Y(\sigma(T)-\tau(T))\ra \O_\Y(\sigma(T))\ra \O_{\tau(T)}(\sigma(T))\ra 0.
\end{align}
Pushing \eqref{E:bundle-ses} via $\pi$ to $T$, we obtain a split
short exact sequence
\[
0 \ra \L \ra \E \ra \O_T\ra 0,
\]
where $\L\simeq \tau^*\O_\Y(-\tau(T))\simeq \sigma^*\O_\Y(\sigma(T))$ and 
$\E=\pi_*\O_\Y(\sigma(T))$.
It follows that $\Y=\PP(\E)=\PP(\O_T\oplus \L)$.
The divisor $D$ is
the vanishing locus of a section $s_D\in \HH^0(\Y, \O_\Y(D))$. Using $D\sim (n+1)\sigma(T)$, we rewrite  
\begin{align*}
\HH^0(\Y, D)=\HH^0(T, \pi_*\O_\Y(D))  
=\HH^0(T, \Sym^{n+1}(\O_T\oplus \L))=\bigoplus_{m=0}^{n+1} \HH^0(T, \L^m).
\end{align*}
Hence, $D$ defines a section of $\L, \L^2, \dots, \L^{n+1}$ each. It is easy to check that the section of $\L$ defined by $D$ is actually $0$. We conclude that a $\W$\nb-stable
cover defines a line bundle $\L$ and sections $s_m\in \HH^0(T, \L^m)$ for each $m=2,\dots, n+1$.
Thus, there is an equivalence between the category of $\W$-stable
covers and $\P(2,3,\dots, n+1)$ in the case of odd $n$. 

The case of $n$ even is treated similarly: The only minor difference is 
that the section at infinity $\tau$ is now odd. Given 
a $\W$\nb-stable cover $(\varphi\co \X\ra \Y; \tau, D)$, the
section $\tau(T)$ lies in the branch locus of the map 
$\phi\co X\ra Y$ between the coarse moduli spaces. 
As in the case of odd $n$, $Y$ is a $\PP^1$\nb-bundle and
the 
existence of the 
divisor $D\subset Y$ gives a section $\sigma\co T\ra Y$ disjoint from 
$\tau(T)$.
We have 
$\pi_*\O_Y(\sigma(T))=\O_T\oplus \L$. As before, the divisor
$D$ defines sections $s_m \in \HH^0(T,\L^m)$ for each 
$m=2,\dots, n+1$. It remains to observe that by our construction
the line bundle
$\L\simeq \tau^*\O_Y(-\tau(T))$ has a square root because
$\O_Y(\tau(T))$ pulls back to the line bundle $\O_X(2\tau_1(T))$
via $\phi$. Hence, we can write $\L\simeq \mathcal{M}^2$ and 
we conclude that giving a family of $\W$\nb-stable covers over $T$
is equivalent to giving a line bundle $\mathcal{M}$ on $T$ and 
sections $s_m\in \HH^0(T, \mathcal{M}^{2m})$ for each $m=2,\dots, n+1$. This shows that $\H_{n}[n-1]\simeq \P(4, 6, \dots, 2(n+1))$.

We remark that  in positive characteristic the stack 
$\H_{n}[n-1]$ can have nonreduced stabilizers, as Example \ref{E:infinite-automorphisms} shows. Furthermore, in positive characteristic,
$\H_{n}[n-1]$ is in general not isomorphic to $\P(2,3,\dots,n+1)$ or
$\P(4, 6, \dots, 2(n+1))$.
We refer to \cite{arsie-vistoli} for the 
description of the stack of degree $2$ cyclic covers of $\PP^1$ in this case. 

Finally, we briefly sketch what happens in the case of $\H_{n}[n-1,n-1]$. Take $\beta=\alpha=1/n$, and for $\W=(1,\beta, \alpha^{n})$,
consider a $\W$\nb-stable pointed quasi-admissible cover
$(\varphi\co \X\ra \Y; \tau, \chi, D)$. 
The coarse moduli space $Y\ra T$ is a $\PP^1$\nb-bundle
with a section $\sigma:=\varphi\circ\chi\co T\ra Y$ 
disjoint from $\tau(T)$ and such 
that $D\sim n\sigma$. Moreover, we have $\pi_*\O_Y(\sigma(T))=\O_T\oplus\L$, where $\L\simeq \tau^*\O_Y(-\tau(T))\simeq \sigma^*\O_Y(\sigma(T))$. As before, we see that $D$ defines sections
$s_m\in \HH^0(T,\L^m)$ for $m=1,\dots,n$ and that $\L$ has 
a square root when $n$ is odd. It remains to observe that 
in the presence of the section $\chi\co T\ra X$, 
the line bundle $\sigma^*\O_Y(D)$ has a square root because
$D$ pulls back via $\varphi$ 
to the square of the ramification divisor 
$\text{Ram}(\varphi)$.
Since $\L^{n}\simeq \sigma^*\O_Y(n\sigma(T))\simeq \sigma^*\O_Y(D)$,
we conclude that $\chi^*(\text{Ram}(\varphi))$ gives us a section 
of the square root of $\L^{n}$. It is now clear that 
\[
\H_{n}[n-1,n-1] =\begin{cases} \P(\frac{n}{2}, 1,2,3,\dots,n-1), &\text{ if $n$ is even,}\\
 \P(n, 2,4,6,\dots,2n-2), &\text{ if $n$ is odd}.
 \end{cases}
\]

We note that the explicit presentation of $\H_{n}[n-1]$ given above 
in the case of odd $n$ has an added advantage of proving that 
$\H_{n}[n-1]\simeq \left[ \Def(A_n) \smallsetminus \mathbf{0} \ / \GG_m\right]$, where the $\GG_m$ action is given by Equation \eqref{E:A-action}. 
Similar arguments apply to give the following result whose proof we omit. 
\begin{corollary}\label{C:pinkham-blowup}
For $\GG_m$ actions on $\Def(A_n)$ and $\Def(D_n)$ given in Equations \eqref{E:A-action} --\eqref{E:D-action}, 
\begin{enumerate}
\item $\H_{n}[n-1]\simeq \left[ \Def(A_n) \smallsetminus \mathbf{0} \ / \GG_m\right]$,
\item $\H_{n}[n-1,n-1]\simeq \left[ \Def(D_n) \smallsetminus \mathbf{0} \ / \GG_m\right]$.
\end{enumerate}
\end{corollary}
\end{example}
 The following two theorems establish 
that $\H_{n}[k]$ and $\H_{n}[k,\ell]$ are proper 
and smooth Deligne-Mumford stacks over $\mathbb{K}$.
\begin{theorem} 
\label{T:DM}
The stacks $\H_{n}[k]$ and $\H_{n}[k,\ell]$ are proper Deligne-Mumford stacks over $\mathbb{K}$.
\end{theorem}
\begin{proof}
Evidently, the forgetful natural transformation 
$\H_{n}[k,\ell] \ra \H_{n}[k]$ is 
representable by proper Deligne-Mumford stacks. 
Thus we reduce to the case of $\H_{n}[k]$. 
By Definition \ref{D:stack-H}, the branch morphism $\brr\co \H_{n}[k]\ra \Mdiv{\W}$ 
factors through $\Mdiv{\W}^{\ev}$. Given a family of even rational orbicurves,
there exists a unique quasi-admissible cover over it (with $\mu_{2}$ of
extra automorphisms if $n$ is even). 
Thus, the morphism $\H_{n}[k]\ra \Mdiv{\W}^{\ev}$ is representable by 
proper Deligne-Mumford stacks.
Finallly, $\Mdiv{\W}^\ev$ is a proper Deligne-Mumford 
stack over $\mathbb{K}$ by Theorem \ref{T:even-orbicurves}. This finishes the proof.
\end{proof}

\subsection{Deformation theory}

\begin{theorem} 
\label{T:H-smooth-stack}
$\H_{n}[k]$ and $\H_{n}[k,\ell]$ are smooth Deligne-Mumford stacks 
over $\mathbb{K}$.
\end{theorem}


Consider a pointed quasi-admissible cover over $T=\spec \mathbb{K}$ as in Display \eqref{E:stack-H1}. 
Note that the divisor $D$ is uniquely 
determined by $\varphi$. 
Similarly, when $\tau$ is even, the section $\tau_1$ determines $\tau_2$. 
Therefore, we reduce to the study of the deformation theory of 
\begin{align}
\label{E:def}
\xymatrix{
\X
\ar[r]^{\varphi}  & \Y \ar[d]_{\pi} &  \\ 
& \spec T\ar@/_0.5pc/[ul]^{\chi} \ar@/^1pc/[ul]^{(\tau_1)} 
\ar@/_1pc/[u]_{\tau}
}
\end{align}
Note that the deformations of the closed immersions 
$\tau_1\co \spec \mathbb{K}\ra \X$, if $\tau$ is even, and of 
$\tau\co \spec \mathbb{K}\ra \Y$, if $\tau$ is odd, are unobstructed
because $\tau$ lies in $\Y^{\smooth}\smallsetminus D$ by Definition \ref{D:stack-H}.
It follows that the obstructions to deforming Diagram \eqref{E:def} 
vanish as long as the diagram
\begin{align}
\label{E:def-nosection}
\xymatrix{
\X
\ar[r]^{\varphi}  & \Y \ar[d]_{\pi} &  \\ 
& \spec T\ar[ul]^{\chi} }
\end{align}
 is unobstructed. 
 
\def\mi{\mathfrak{m}}

Suppose now that 
$(\varphi\co \X \ra \Y; \chi)$ is a cover as in 
Diagram \eqref{E:def-nosection} over $T=\spec A$, where
$A$ is a 
local Artinian $\mathbb{K}$\nb-algebra. Set 
$D=\brr(\varphi)$. We say that an open affine cover 
$\UU=\{U_i\}_{i\in S}$ of $\Y$ 
is {\em adapted} if every point of $D+ \varphi(\chi)+\Sing(\Y)$
is contained in exactly one $U_i$. 

\begin{lemma}
\label{L:versality}
 Let $(\varphi\co \X \ra \Y; \chi)$ be a quasi-admissible 
cover over a local Artinian $\mathbb{K}$-algebra $A$, and
$\UU=\{U_i\}_{i\in S}$ be an adapted affine cover of $\Y$. Consider 
a $\mathbb{K}$-extension $0\ra I \ra A' \ra A \ra 0$.
Then any collection of deformations of $\varphi\vert_{U_i}$ 
over $A'$, for every $i\in S$,
gives rise to a global deformation of $(\varphi\co \X \ra \Y; \chi)$ over $A'$.
\end{lemma}
\begin{proof}
By the assumption, the intersections $U_i\cap U_j$ are smooth affine 
schemes and so have only trivial deformations. It follows that there is an
isomorphism $g_{ij}\in \HH^0(U_{ij}, T^1\otimes I)$ between the restrictions to $U_{ij}$
 of deformations over $U_i$ and $U_j$; here, $T^1=\SHom(\Omega^{1}_Y,\O_Y)$. The vanishing of the associated $2$-cocycle  
 in $\HH^2(Y,T^1 \otimes I)$
 is a necessary and sufficient condition for the existence of the requisite 
 global extension of
 $(\varphi\co \X \ra \Y; \chi, D)$  to $A'$. But $\HH^2(Y,T^1\otimes I)=0$ because 
$Y$ is proper of dimension one.
\end{proof}

\begin{prop}\label{P:versality} Let $\varphi\co \X\ra\Y$ be a quasi-admissible
cover over $\spec \mathbb{K}$.
 For any adapted cover $\UU=\{U_i\}_{i\in S}$ of $\Y$
the map 
\[
\Def(\varphi) \ra \prod_{i\in S} \Def(\varphi\vert_{U_i})
\]
is formally smooth.
\end{prop}
\begin{proof}
This is Lemma \ref{L:versality}.
\end{proof}
\begin{proof}[Proof of Theorem \ref{T:H-smooth-stack}:]
Let $(\varphi\co \X \ra \Y; \tau, \chi, D)$ be a pointed quasi-admissible cover over 
$\mathbb{K}$. By Proposition \ref{P:versality}, we need only to check that for every 
affine chart $U$ in an adapted open cover $\UU=\{U_i\}_{i\in S}$ of $\Y$, the 
deformation space $\Def(\varphi\vert_{}{U_i})$ is formally smooth.
The analysis breaks down into several cases.\par
\noindent
{\em Case {\rm I:} $U$ contains an (orbi)node.} 
The module of K\"{a}hler differentials of $B=A[x,y]/(x,y)$ has a two-term free 
resolution 
$ 0 \ra B \stackrel{\binom{y}{x}}{\ra} B\oplus B\ra \Omega^1_{B/A}\ra 0$. 
As observed in \cite{AV}, the above resolution 
is $\mu_2$-equivariant for the action $(x,y) \ra (-x,-y)$. It follows that locally at 
the (orbi)node
the projective dimension of $\Omega^1_{U}$ is $2$. In particular, 
$\Ext^2(\Omega^1_U,\O_U)=0$ and $\Def(U)$ is smooth.
It remains to observe that since $\varphi$ is finite \'{e}tale
over $U$, the map $\Def(\varphi\vert_{U})\ra \Def(U)$ is formally smooth.
\par
\noindent
{\em Case {\rm II:} $U$ is smooth and contains $\varphi(\chi)$.}
Let $U=\spec A[x]$ and $\varphi^{-1}(U)=\spec A[x]/(y^2-f(x))$. 
Without loss of generality, 
the section $\chi$ is $(x=0, y=a_0)$, where $a_{0}^{2}=f(0)$. 
For an arbitrary $\mathbb{K}$-extension $A' \ra A\ra 0$, we would like to find a lifting of $f(x)$ 
and $a_0$ to $A'[x]$ and $A'$, respectively. Write $f(x)=xg(x)+a_0^2$. Take $g'(x)$ to be an arbitrary lifting of $g(x)$
and $a_0'$ be an arbitrary lifting of $a_0$. Then $f'(x)=xg'(x)+(a'_0)^2 \in A'[x]$ is a lifting of $f(x)$ to $A'[x]$ and the section
$\chi': (x=0, y=a'_0)$ is a lifting of $\chi$.

\par  \noindent
{\em Case {\rm III:}  $U$ is smooth and does not contain $\varphi(\chi)$.}
This is the simplest case: the deformations of $\varphi\vert_{U}$ are the same
as deformations of $U$ together with the divisor $D\cap U$. The deformation space
is smooth because $U$ is a smooth one dimensional affine scheme.
\end{proof}

Propostion \ref{P:A=D} of Section \ref{S:D-singularities} establishes the equivalence  between deformations of an $A_{k-1}$ singularity with a 
section and a $D_k$ singularity. Using it, we rephrase Proposition 
\ref{P:versality} as follows. 
\begin{prop}\label{P:versality2} Let
 $(\varphi\co \X \ra \Y; \tau, \chi, D)$ be a pointed quasi-admissible cover. Suppose that $D=k_0p_0+\sum_{i=1}^r (k_i+1)p_i$, where $p_i$
are distinct points of $\Y$. Further assume that $\varphi(\chi)=p_0$. Then the natural morphism
\[
\Def(\varphi) \ra \Def(D_{k_0})\times \Def(A_{k_1})\times\cdots \times \Def(A_{k_r})
\]
is formally smooth.
\end{prop}

\section{The log minimal model program for $\Tl{A_n}$ and $\Tl{D_n}$}
\label{S:logMMP}

In this section, we prove Parts 5 of Main Theorems \ref{main1} and \ref{main2}.
The key part of the argument is the 
positivity of certain log
canonical divisors on $\H_{n,\alpha}$ and $\H_{n,\alpha,\beta}$, which we establish
using \cite{FedAmple}. Throughout, we use standard
notation for the divisor classes on moduli spaces of pointed curves.

We begin by introducing
natural Cartier divisor classes on $\H_{n,\alpha}$ and $\H_{n,\alpha,\beta}$,
which we define in terms of divisor classes on moduli spaces of weighted
pointed curves. 
Given a weight vector $\A=\bigl(1,\underbrace{\alpha,\dots,\alpha}_{n+1})$ 
or $\A=\bigl(1,\beta, \underbrace{\alpha,\dots,\alpha}_{n})$, 
consider the moduli space $\Mg{0,\A}$ of $\A$\nb-stable rational curves \cite{Hweights}. 
We recall that irreducible boundary divisors on $\Mg{0,\A}$ are parameterized by partitions $I\cup J=\A$. We
arrange $I$ to always contain weight $1$. 
Set
\begin{align*}
\Delta_{\odd} &:=\sum_{\vert J\vert \text{ is odd and $\sum_{i\in J} \alpha_i>1$}}\Delta_{I,J} \, , 
& \Delta_s &:=\sum_{\vert J \vert =2, \ \sum_{i\in J} \alpha_i\leq 1}\Delta_{I,J} \, , \\
\Delta_{\ev} &:=\sum_{\vert J\vert \text{ is even and $\sum_{i\in J} \alpha_i>1$}} \Delta_{I,J} \, , 
& \Delta&:= \Delta_{\odd}+\Delta_{\ev} \, .
\end{align*}
When $\A=\bigl(1,\beta,\alpha,\dots,\alpha)$, 
we also denote by $\Delta_{\sigma\cap \chi}$
the divisor of curves 
$(C; \tau, \chi, \{\sigma_{i}\}_{i=1}^{n})$ in $\Mg{0,\A}$ such that $\sigma_{i}=\chi$ for some $i$.
In other words, $\Delta_{\sigma\cap \chi}$ parameterizes curves where the point of weight $\beta$ coincides with one of the points of weight $\alpha$. 
When $\alpha+\beta>1$, we take $\Delta_{\sigma\cap\chi}=0.$ 

Let $\q\co \Mg{0,\A} \ra \Mg{0,\A}/\S$ be the quotient map, where 
$\S=\S_{n+1}$ if $\A=(1,\underbrace{\alpha,\dots,\alpha}_{n+1})$, and
$\S=\S_n$ if $\A=(1,\beta, \underbrace{\alpha,\dots,\alpha}_{n})$. For 
$\W=(1,\alpha^{n+1})$ or $\W=(1,\beta,\alpha^{n})$, we abuse notation and
denote $\Mg{0,\A}\ra \R_{\W}$ also by $\q$ (cf. Section \ref{S:different-marked-curves}). This abuse
is inconsequential to the study of divisors, since 
the coarse moduli map 
$\Mdiv{\W} \ra \Mg{0,\A}/\S$ is an isomorphism in codimension one. 
Given $D\subset \Mg{0,\A}$, we let $\q(D)$ be the reduced image. 
We also conflate the branch morphism $\brr$ and its composition 
with $\Mdiv{\W} \ra \Mg{0,\A}/\S$.

\subsection{Boundary divisors on $\H_{n,\alpha}$ and $\H_{n,\alpha,\beta}$} 
We let $\delta_{\irr}$ to be the divisor (on $\H_{n,\alpha}$ or $\H_{n,\alpha,\beta}$)
parameterizing covers $\X \stackrel{\varphi}{\ra} \Y$ such 
that $\X$ has a singularity lying over a smooth point of $\Y$. 
Equivalently, a cover is in $\delta_{\irr}$ if the branch divisor has a point
of multiplicity at least $2$. It follows that 
\begin{align*}
\delta_{\irr}=\brr^*\q(\Delta_s).
\end{align*}

We let $\delta_{\red}$ to be the divisor parameterizing reducible covers 
$\X \stackrel{\varphi}{\ra} \Y$. 
A cover is reducible if and only if $\Y$ has an (orbi)node. 
In terms of divisors on $\Mg{0,\A}/\S$ 
we have that 
\begin{align*}
\delta_{\red}=\brr^{*}\q(\Delta_\ev)+\frac{1}{2}\brr^{*}\q(\Delta_{\odd}).
\end{align*}

On $\H_{n,\alpha,\beta}$ we define the {\em Weierstrass
divisor} $\delta_{W}$ to be the locus of
$(\varphi\co \X\ra\Y; \tau,\chi, D)$ such that $\chi$ is a ramification point 
of $\varphi$. Equivalently,
\begin{align*}
\delta_{W}=\br^*\q(\Delta_{\sigma\cap \chi}).
\end{align*}

\subsubsection{Log canonical divisors on $\H_{n,\alpha}$}
\label{S:canonical-1}

Let $\W=(1,\alpha^{n+1})$. 
Recall that $\Mdiv{\W}$ is the moduli space of $\W$\nb-stable rational curves with a section at infinity 
of weight $1$ and a divisor of degree $n+1$ and weight $\alpha$ (see Section 
\ref{S:moduli-spaces}).
Consider the commutative diagram
\begin{align}\label{D:main-diagram}
\xymatrix{
& \H_{n,\alpha} \ar[d]^{\brr} \ar[drr]^{\brr} \ar[rr]^(0.4){\text{coarse}} & & H_{n,\alpha} \ar[d]^{\simeq} \\
\Mg{0,\A} \ar[r]^{\q} 
& \Mdiv{\W} \ar[rr]^(0.4){\text{coarse}} & & \Mg{0,\A}/\S_{n+1}=R_{\W}.
}
\end{align}


\begin{lemma}\label{L:ramification} 
\hfill

\begin{enumerate}
\item[(a)]The morphism $\q\co \Mg{0,\A} \ra \Mg{0,\A}/\S_{n+1}$  is simply ramified 
over $\Delta_s$.
\item[(b)]The  morphism $\brr\co \H_{n,\alpha}\ra \Mg{0,\A}/\S_{n+1}$ is simply ramified over $\q(\Delta_{\odd})$.
\end{enumerate}
\end{lemma}
\begin{proof} We note that (a) is standard: the quotient morphism is always ramified over the locus where the stabilizer is nongeneric (in our case, non-trivial).
A local computation shows that the ramification is simple.

To prove (b), observe that the morphism $\brr$ is ramified along the codimension one locus of hyperelliptic covers having an extra automorphism. These are precisely reducible covers possessing an odd node. By our convention, such covers map to the divisor $\q(\Delta_{\odd})$. 
\end{proof}

\begin{prop}\label{P:canonical-class-1}
The canonical divisor of $\H_{n,\alpha}$ is $K_{\H_{n,\alpha}}=\brr^{*}(D)$, where
$D$ is the divisor on $R_\W=\Mg{0,\A}/\S_{n+1}$ satisfying
\begin{align*}
\q^{*}(D)=\psi-\Delta_s-2\Delta_\ev-\frac{3}{2}\Delta_{\odd}
\end{align*}
\end{prop}
\begin{proof}
By Lemma \ref{L:ramification}(b), we have 
$
K_{\H_{n,\alpha}}=\brr^*\left(K_{R_\W}+\frac{1}{2}\q(\Delta_\odd)\right)
$.
By Lemma \ref{L:ramification}(a), we have
$\q^*\left(K_{R_\W}+\frac{1}{2}\q(\Delta_\odd)\right)=K_{\Mg{0,\A}}-\Delta_s+\frac{1}{2}\Delta_\odd$.
It remains to recall that by the Grothendieck-Riemann-Roch formula 
(see also \cite[Section 3.1.1]{Hweights}):
$$K_{\Mg{0,\A}}=\psi_\sigma+\psi_\tau-2\Delta=\psi-2\Delta_{\ev}-2\Delta_{\odd}.$$
\end{proof}

\subsubsection{Log canonical divisors on $\H_{n,\alpha,\beta}$}
\label{S:canonical-2}
The canonical divisor of $\H_{n,\alpha,\beta}$ is computed analogously. 
Let $\W=(1, \beta, \alpha^{n})$ and $\A=(1, \beta, \underbrace{\alpha,\dots,\alpha}_{n})$, and 
consider the diagram
\begin{align}\label{D:main-diagram-2}
\xymatrix{
& \H_{n,\alpha,\beta} \ar[d]^{\brr} \ar[drr]^{\brr} \ar[rr]^(0.4){\text{coarse}} & & H_{n,\alpha,\beta} \ar[d]^{\simeq} \\
\Mg{0,\A} \ar[r]^{\q} 
& \Mdiv{\W} \ar[rr]^(0.4){\text{coarse}} & & \Mg{0,\A}/\S_{n}=R_{\W}.
}
\end{align}The minor difference is 
that the morphism
$\brr\co \H_{n,\alpha,\beta} \ra R_{\W}$ is now ramified 
over $\q(\Delta_\odd)$ 
and $\q(\Delta_{\sigma\cap\chi})$.  Therefore, 
$ 
K_{\H_{n,\alpha,\beta}}=\brr^*\left(K_{R_{\W}}+\frac{1}{2}\q(\Delta_\odd+\Delta_{\sigma\cap \chi})\right).
$
The quotient morphism $\q\co \Mg{0,\A} \ra \Mg{0,\A}/\S_n$ is still ramified only at $\Delta_s$.
It follows that
\begin{multline*}
\q^*\left( K_{R_{\W}}+\frac{1}{2}\q(\Delta_\odd+\Delta_{\sigma\cap \chi})\right) =
K_{\Mg{0,\A}}-\Delta_s+\frac{1}{2}(\Delta_\odd+\Delta_{\sigma\cap \chi}) \\
=\psi_\tau+\psi_\sigma+\psi_\chi-\Delta_s-2\Delta_\ev-\frac{3}{2}\Delta_{\odd}+\frac{1}{2}\Delta_{\sigma\cap \chi}.
\end{multline*}
We summarize these computations in the following proposition.
\begin{prop}\label{P:canonical-class-2}
The canonical divisor of $\H_{n,\alpha,\beta}$ is $K_{\H_{n,\alpha,\beta}}=\brr^{*}(D)$, where
$D$ is the divisor on $R_\W=\Mg{0,\A}/\S_{n}$ satisfying
\begin{align*}
\q^{*}(D)=\psi-\Delta_s-2\Delta_\ev-\frac{3}{2}\Delta_{\odd}+\frac{1}{2}\Delta_{\sigma\cap\chi}.
\end{align*}
\end{prop}

\subsection{Log canonical models of $\Tl{A_n}$}

Fix an integer $n\geq 2$. In Section \ref{S:reduction}, we show that for $k'\geq k$ there exists a natural reduction morphism 
$\H_{n}[k] \ra \H_{n}[k']$. 
If $k'$ is even (resp. odd), the morphism $\H_{n}[k]\ra\H_{n}[k']$ replaces a tail of 
genus $k'/2$ 
(resp. a bridge of genus $(k'-1)/2$) 
 by an $A_{k'}$ singularity. We obtain a sequence of birational contractions
\[
\H_{n}[1]\ra \H_{n}[2]\ra \dots \ra \H_{n}[n-1],
\]
where the initial space 
$\H_{n}[1]$ is 
the stack of Harris-Mumford {\em admissible covers} of genus 
$\left \lfloor \frac{n}{2}\right \rfloor$ 
(recall that admissible covers have at worst $A_1$ singularities), and, by 
Example \ref{E:weighted-projective-spaces},
the final space is 
\[
\H_{n}[n-1]\simeq \left[\Def(A_n)\smallsetminus \mathbf{0} \ / \ \GG_m\right]
\simeq\begin{cases} \P(2,3,\dots, n+1),\text{ if $n$ is odd}, \\ 
\P(4,6,\dots, 2n+2),\text{ if $n$ is even.}
\end{cases}
\]
(It is important to note that the generic stabilizer of $\H_n[n-1]$ is $\mu_2$ in the case of even $n$.) 

In what follows, we describe the intermediate
spaces as log canonical models of $\H_n[1]$. 
\begin{prop}\label{P:ampleness-H} 
The divisor $K_{\H_{n,\alpha}}+(\alpha+1/2)\delta_\irr+\delta_\red$
is a pullback of an ample divisor from the coarse moduli space. 
\end{prop}
\begin{proof}
By Proposition \ref{P:canonical-class-1}, we have $K_{\H_{n,\alpha}}+(\alpha+1/2)\delta_\irr+\delta_\red=\brr^*(D)$, where $D$ is the divisor
on $\Mg{0,\A}/\S_{n+1}$ satisfying 
\begin{align*}
\q^*(D) &=(\psi_\tau+\psi_\sigma-\Delta_s-2\Delta_\ev-\frac{3}{2}\Delta_\odd)+(2\alpha+1) \Delta_s+(\Delta_\ev+\frac{1}{2}\Delta_\odd) \\
&= \psi_\tau+\psi_\sigma+2\alpha\Delta_s-\Delta_n.
\end{align*}
By \cite[Theorem 4]{FedAmple}, the divisor $\psi_\tau+\psi_\sigma+2\alpha\Delta_s-\Delta_n$ is ample on $\Mg{0,\A}$.
\end{proof}
We obtain the following result which finishes the proof of Part 5 of Main Theorem \ref{main1}.
\begin{theorem}\label{T:logMMP-A} Let $k\in\{1,\dots,n-1\}$, then for any $\alpha\in  \left(\frac{1}{2}+\frac{1}{k+2}, \frac{1}{2}+\frac{1}{k+1}\right] \cap \QQ$ 
\[
H_{n}[k]=\proj R\left( \H_n[1], K_{\H_n[1]}+\alpha\delta_\irr+\delta_\red\right).
\]
In other words, the coarse moduli space of $\H_{n}[k]$ is a log canonical model of $\H_{n}[1]$.
\end{theorem}
\begin{proof}
By Proposition \ref{P:ampleness-H}, the divisor $K_{\H_{n,\alpha}}+\alpha\delta_\irr+\delta_\red$ is a pullback of an ample divisor from the coarse
moduli space. Let $f\co \H_{n}[1]\ra\H_{n}[k]$, then the discrepancy
$$
K_{\H_{n}[1]}+\alpha\delta_{\irr}+\delta_{\red}-f^*(K_{\H_{n}[k]}+\alpha\delta_{\irr}+\delta_{\red})
$$
is effective for $\alpha\leq  \frac{k+3}{2(k+1)}$ by Lemma \ref{L:discrepancy}.
The thesis now follows from \cite[Proposition A.13]{HH1}.
\end{proof}
\subsection{Log canonical models of $\Tl{D_n}$}

Fix a number $n\geq 4$. In Section \ref{S:reduction}, we construct a
natural reduction morphism  $\H_n[k,\ell]\ra \H_n[k',\ell']$ (for 
$k'\geq k$ and $\ell'\geq \ell$) such that 
\begin{itemize}
\item If $k'$ is even (resp. odd), then any tail of genus $k'/2$ (resp. a bridge of genus $(k'-1)/2$) is replaced 
 by an $A_{k'}$ singularity. 
\item If $\ell'$ is even (resp. odd), then any pointed bridge of genus $\ell'/2-1$ (resp. a pointed tail of genus $(\ell'-1)/2$) is 
replaced by a $D_{\ell'}$ singularity.
 \end{itemize}
We obtain a lattice of moduli stacks, where each arrow is a divisorial contraction:
{\small
\begin{align*}
\xymatrix{
\H_{n}[1,1] \ar[d]   \ar[r]  & \H_{n}[1,2] \ar[d]\\
\H_{n}[2,1] \ar[r] \ar[d] & \H_{n}[2,2]  \ar[d]  \ar[r] & \H_{n}[2,3] \ar[d]  \\
\cdots \ar[r] \ar[d] & \cdots \ar[r] \ar[d] & \cdots \ar[d]  \ar[r] & \cdots\ar[d]\\
\H_{n}[n-2,1] \ar[r] \ar[d] & \H_{n}[n-2,2] \ar[r] \ar[d] & \cdots \ar[d]\ar[r] & \cdots \ar[r] \ar[d]& \H_{n}[n-2,n-1] \ar[d] & \\
\H_{n}[n-1,1] \ar[r]  & \H_{n}[n-1,2] \ar[r] & \cdots \ar[r] & \H_{n}[n-1,n-2] \ar[r] 
& \H_{n}[n-1,n-1]  &
}
\end{align*}
}
\noindent
$\H_{n}[1,1]$ is the stack of 
$1$-pointed admissible covers of genus $\lfloor\frac{n-1}{2}\rfloor$ and
by Example \ref{E:weighted-projective-spaces}
\[
\H_{n}[n-1,n-1]\simeq \left[\Def(D_n) \smallsetminus \mathbf{0} / \GG_m \right]
\simeq
\begin{cases} \P(n/2, 1,2,3,\dots,n-1), &\text{ if $n$ is even,}\\
 \P(n, 2,4,6,\dots,2n-2), &\text{ if $n$ is odd}.
 \end{cases}
 \]
 We now describe the intermediate spaces as log canonical 
 models of $\H_n[1,1]$. For this we look at special log canonical divisors 
 described in the following proposition.
\begin{prop}\label{P:ampleness-H-2} The divisor
$K_{\H_{n,\alpha,\beta}}+(\alpha+1/2)\delta_\irr+\delta_\red+(2\alpha+2\beta-1) \delta_{W}$ is a pullback of an ample divisor from the coarse moduli space. 
\end{prop}
\begin{proof}
By Proposition \ref{P:canonical-class-2}, we have 
$K_{\H_{n,\alpha,\beta}}+(\alpha+1/2)\delta_\irr+\delta_\red+(2\alpha+2\beta-1) \delta_{W}=\brr^*(D)$ where $D$ is the divisor
on $\Mg{0,\A}/\S_{n}$ satisfying 
\begin{align*}
\q^*(D)&=(\psi-\Delta_{s}-2\Delta_{\ev}-\frac{3}{2}\Delta_{\odd}
+\frac{1}{2}\Delta_{\sigma\cap \chi}) +(2\alpha+1)\Delta_s \\ &{\quad}+
(\Delta_{\ev}+\frac{1}{2}\Delta_{\odd}) +\frac{1}{2}(2\alpha+2\beta-1)\Delta_{\sigma\cap \chi}\\
&=\psi_\tau+\psi_\sigma+\psi_\chi-\Delta_n+2\alpha\Delta_s+(\alpha+\beta)\Delta_{\sigma\cap\chi}.
\end{align*}
Applying \cite[Theorem 4]{FedAmple} for the weight vector $\A=(1,\beta, \underbrace{\alpha,\dots,\alpha}_{n})$
we conclude that 
$$\psi_\tau+\psi_\sigma+\psi_\chi-\Delta_n+2\alpha\Delta_s+(\alpha+\beta)\Delta_{\sigma\cap\chi}$$
is ample on $\Mg{0,\A}$. 
\end{proof}
Proposition \ref{P:ampleness-H-2} together with Lemma \ref{L:discrepancy} 
and \cite[Proposition A.13]{HH1} finish the proof of Part 5 of Main Theorem \ref{main2}.
Specializing to the case of $\alpha+\beta=1/2$, 
we obtain the following result.
\begin{theorem}\label{T:logMMP-D} Let $k\in\{2,\dots,2n-4\}$. 
Then $\H_{n,1/k,(k-2)/2k}=\H_{n}[k-1,\ell]$, where $\ell=\lfloor k/2\rfloor +1$. 
In particular, for $\alpha=\frac{1}{2}+\frac{1}{k}$
\[
H_{n}[k-1,\ell]=\proj R\left( \H_n[1,1], K_{\H_n[1,1]}+\alpha\delta_\irr+\delta_\red\right).
\]
\end{theorem}

\begin{remark} Note that by Theorem \ref{T:logMMP-D},
the singularities $A_{k-1}$ and $D_\ell$ appear at the threshold values of $\alpha$ that equal to  
the log canonical thresholds of the discriminants inside deformation 
spaces $\Def(A_{k-1})$ and $\Def(D_\ell)$ (see Section \ref{S:lct}).
\end{remark}


\section{Reduction morphisms}
\label{S:reduction}
In this section, we establish the existence of natural reduction morphisms between stacks 
$\H_{n}[k,\ell]$.
Recall that $\H_{n}[k,\ell]=\H_{n,\alpha,\beta}$ parameterizes pointed quasi-admissible hyperelliptic covers  
$(\varphi\co \X\ra \Y; \tau, \chi, D)$, where the branch divisor $D$ has degree $n$ and weight $\alpha$ and the section $\chi$ has weight 
$\beta$.
The integers $k$ and $\ell$ are determined uniquely by the inequalities \begin{eqnarray}
\label{inequalities2}
\frac{1}{k+2} < \alpha \leq \frac{1}{k+1}, \\ 
1-(\ell+1)\alpha< \beta \leq 1-\ell \alpha. \notag
\end{eqnarray}
The starting point for us will be the existence of reduction morphisms
between moduli stacks of divisorially marked rational curves. Namely,
if $\W=(1,\beta,\alpha^n)$ and $\W'=(1,\beta',(\alpha')^{n})$ are weight vectors such that $\alpha\geq \alpha'$ and $\beta\geq \beta'$, then by
\cite[Theorem 4.1]{Hweights}\footnote{The proof in \cite{Hweights} is given for weighted pointed curves but generalizes word for word to the situation of divisorially marked curves.} there exists the reduction morphism 
$r_{\W,\W'}\co \Mdiv{\W}\ra \Mdiv{\W'}$. 
Before we proceed we make a useful observation.
\begin{lemma}\label{L:increasing}
Let $(k,\ell)$ and $(k',\ell')$ be the integers associated to weight vectors 
$\W$ and $\W'$ and uniquely determined by Inequalities \eqref{inequalities2}. Then 
the reduction morphism 
$r_{\W,\W'}\co \Mdiv{\W} \ra \Mdiv{\W'}$ exists
 as long as $k\leq k'$ and $\ell\leq \ell'$. When this happens we say that 
 $(k,\ell)\preceq (k',\ell')$.
\end{lemma}
\begin{proof} The lemma says that the morphism $r_{\W,\W'}$ exists 
even when $\beta'>\beta$ or $\alpha'>\alpha$ as long as the condition $(k,\ell)\preceq (k', \ell')$ is satisfied.
The idea of proof is clear: The $\W$\nb-stability is a condition on how many points of weight $\alpha$ and $\beta$ can come together and so $\Mdiv{\W}$ does not change as long as integers $k$ and $\ell$ remain constant. We 
establish the existence of the reduction morphism after tweaking the weights in such a way that $\Mdiv{\W}$ and $\Mdiv{\W'}$ are unchanged but 
\cite[Theorem 4.1]{Hweights} now applies.

It suffices to supply the proof for the cases $(k',\ell')=(k+1,\ell)$ and $(k',\ell')=(k,\ell+1)$. In the former case, we take $\alpha=1/(k+2)+\varepsilon$, 
$\alpha'=1/(k+2)$, $\beta=1-\ell\alpha$, $\beta'=1-(\ell'+1)\alpha+\delta$, where $0<\varepsilon, \delta \ll 1$. Then 
$\alpha> \alpha'$ and $\beta> \beta'$, so the reduction morphism exists. In the latter case, we take $\alpha=\alpha'=1/(k+1)$, $\beta=1-\ell\alpha$,
$\beta'=1-(\ell'+1)\alpha$. Then $\alpha= \alpha'$ and $\beta> \beta'$, so the reduction morphism exists.
\end{proof}

\begin{theorem}\label{T:reduction-image} Let $\W=\left(1, \beta, \alpha^{n}\right)$ and 
$\W'=\left(1, \beta', (\alpha')^{n}\right)$ be weight vectors satisfying $(k,\ell)\preceq (k',\ell')$.
Then there exists 
a natural birational {\em reduction} $1$\nb-morphism
\[
r_{\W,\W'}\co \H_{n}[k,\ell] \ra \H_{n}[k',\ell'].
\] 
\end{theorem}

The case of $\H_{n}[k]$ is simpler. We state it here without a proof
for the sake of completeness:
\begin{theorem}\label{T:reduction1} Let $\W=\left(1, \alpha^{n+1}\right)$ and 
$\W'=\left(1, (\alpha')^{n+1}\right)$ be the weight vectors 
with $\alpha\in (1/(k+2), 1/(k+1)]$ and $\alpha' \in (1/(k'+2), 1/(k'+1)]$, 
where $k'\geq k$.
Then there exists 
a natural birational {\em reduction} $1$\nb-morphism
\[
r_{\W,\W'}\co \H_{n}[k] \ra \H_{n}[k'].
\]
\end{theorem}

\begin{proof}[Proof of Theorem \ref{T:reduction-image}:]
 As in Lemma \ref{L:increasing}, we can assume that either 
  $(k',\ell')=(k+1,\ell)$ or $(k',\ell')=(k,\ell+1)$. Moreover, 
  we can tweak weight vectors $\W$ and $\W'$ so that $\alpha\geq \alpha'$ and $\beta\geq \beta'$ without changing $\H_n[k,\ell]$ and $\H_n[k',\ell']$.
  
 Denote $\H:=\H_{n}[k,\ell]$ and $\H':=\H_{n}[k',\ell']$.
As $\H$ is a smooth proper Deligne-Mumford stack over $\mathbb{K}$ 
by Theorems \ref{T:DM} and \ref{T:H-smooth-stack}, 
there exists a surjective smooth morphism $T\ra \H$ from a smooth scheme $T$. We will construct a morphism $f\co T\ra \H'$
such that for $T\times_{\H}T \stackrel{\pr_1,\, \pr_2}{\longrightarrow} T$
 the two morphisms $f\circ \pr_1$
 and $f\circ \pr_2$ are
isomorphic and the isomorphism satisfies the cocycle conditions on 
$T\times_{\H}T\times_{\H}T$. By descent,
this will define a $1$-morphism $\H \ra \H'$. 

 
 To begin, let $(\varphi\co \X\ra \Y; \tau, \chi, D)$ be the exhausting family over 
 $T$. By Lemma \ref{L:splitting}, $\X\simeq \Spec_{\Y}\O_{\Y}\oplus \L^{-1}$,
 where $\L^{2}=\O_{\Y}(D)$. 
Let $\rho\co \Y\ra Y$ be the morphism to the coarse moduli space. 
By definition, $Y$ is a family of $\W$\nb-stable divisorially marked rational curves.
By \cite[Theorem 4.1]{Hweights}, we can form a $T$\nb-morphism $\xi\co Y\ra Y'$
such that $Y'$ is a $\W'$\nb-stable rational curve. 
Moreover, the formation of $\xi$ commutes with base change. 
Briefly, the reduction morphism $\xi$ contracts all rational tails in the fibers of $Y\ra T$ on which the divisor 
$\omega_{Y/T}(\tau+\alpha'D+\beta'\chi)$ has non-positive degree to smooth points in the fibers of $Y' \ra T$. 
Denote by $\E$ the union of the contracted curves.
Then $\xi$ is an isomorphism away from $\E$. 

Set $U:=Y'-\Sing(Y'/T)$ and $V:=Y'-\xi(E)$. Then $U$ and $V$ form an open cover of $Y'$. 
Since $\rho\co \rho^{-1}(V) \ra V$ is an isomorphism, we can construct a $T$-orbicurve 
$\Y'$ as the union of $U$ and $\rho^{-1}(V)$ along $U\cap V$. 
In words, $\Y'$ has the same stack structure as $\Y$ away
from $\xi(\E)$ and is isomorphic to $Y'$ in a Zariski neighborhood of $\xi(\E)$. The $1$\nb-morphism
$\Y'\ra \Y$ is also denoted by $\xi$. 

Since $T$ is smooth, we have that 
$\Y'$ is a smooth scheme in the neighborhood of 
$\xi(\E)$. Moreover,
we have
$\codim(\xi(\E),\Y')\geq 2$ and $\codim(\Sing(\Y'/T), \Y')\geq 2$. Therefore the 
line bundle $\L$ and the Cartier divisor $D$ uniquely extend
 from $U\cap V$ to $\Y'$. Denote these extensions by 
$\L'$ and $D'$. Clearly, $(\L')^{\otimes 2}=\O_{\Y'}(D')$.
The compositions
$\tau'=\xi\circ \tau\co T\ra \Y'$ and  $\chi'=\xi\circ\chi\co T\ra \Y'$ are sections.

By construction, $\Y'$ marked by $\tau', \chi'$ and $D'$ is a $\W'$-stable even 
rational orbicurve over $T$. 
Since the formation of $Y'$ commutes with base change, so does the formation of $\Y'$. Same
holds for the formation of $\tau'$ and $\chi'$. Finally, the formation of $D'$ commutes
with {\em smooth} base change. It follows that $(\Y'; \tau', \chi', D')$ 
is a $\W'$-stable 
even rational orbicurve and there is an isomorphism $\pr_1^*\Y' \simeq \pr_2^*\Y'$ 
of even rational orbicurves over $T\times_{\H} T$ whose pullbacks satisfy the cocycle 
conditions on $T\times_{\H}T\times_{\H}T$.

We now define $\X':=\Spec_{\Y'} (O_{\Y'}\oplus (\L')^{-1})$. From the 
construction,
$\xi_*(\O_\Y\oplus \L^{-1})=O_{\Y'}\oplus {\L'}^{-1}$.
It follows that $\xi_*(\varphi_*\O_\X)=O_{\Y'}\oplus (\L')^{-1}$.
The induced morphism $\X\ra \X'$ is such that  
$\X\ra \X' \ra \Y'$ is the Stein factorization of $\X\ra \Y'$. Clearly, $\X'\ra \X$ contracts those
components in the fibers of $\X \ra T$ which cease to be $\W$-stable.
Finally, if $\tau$ is even, we define sections $\tau_{1},\tau_{2}\co T\ra \X'$ as compositions
of $\tau_{1},\tau_{2}\co T\ra \X$ and $\X\ra \X'$. This finishes the construction 
of the family of quasi-admissible covers over $T$ inducing a morphism
$f\co T\ra \H'$. Since all of the steps in the construction are canonical and commute with 
smooth base change we conclude that $f$ descends to a morphism 
$\H \ra \H'$, as required.
\end{proof}

\subsection*{Local structure of reduction morphisms}  We focus on the case
of $f\co \H_{n}[k,\ell] \ra \H_{n}[k+1,\ell]$ and $g\co\H_{n}[k,\ell] \ra \H_{n}[k,\ell+1]$, the remaining cases being analogous.   

Note that $f \co \H_{n}[k,\ell]\ra \H_{n}[k+1,\ell]$ contracts the divisor defined as the closure of the locus of reducible quasi-admissible covers 
$(\varphi\co \X\ra \Y; \tau, \chi, D)$ such that $Y$ has a rational component $R$ meeting the rest of $Y$ in a single point and such that 
$\deg D\vert_{R}=k+2$. If $k$ is odd, the component of $X$ lying over $R$ is a hyperelliptic curve of genus $(k+1)/2$ meeting the rest
of $X$ in two conjugate points; we say that the cover has a {\em tail} of genus $(k+1)/2$. If $k$ is even, 
the component of $X$ lying over $R$ is a hyperelliptic curve of genus $k/2$ meeting the rest
of $X$ in a single Weierstrass point; we say that the cover has a {\em bridge} of genus $k/2$. Evidently, there is an isomorphism 
\[
\Exc(f) \simeq \H_{k+1}[k, \ell].
\]

The morphism $g \co \H_{n}[k,\ell]\ra \H_{n}[k,\ell+1]$ contracts the divisor defined as the closure of the locus of reducible quasi-admissible covers
$(\varphi\co \X\ra \Y; \tau, \chi, D)$ such that $Y$ has a rational component $R$ meeting the rest of $Y$ in a single point and such that 
$\deg D\vert_{R}=\ell+1$ and $\varphi(\chi)\in R$.
 If $\ell$ is odd, the component of $X$ lying over $R$ is a hyperelliptic curve of genus 
$(\ell-1)/2$ meeting the rest
of $X$ in two conjugate points; we say that the cover has a {\em pointed bridge} of genus $(\ell-1)/2$. If $\ell$ is even, 
the component of $X$ lying over $R$ is a hyperelliptic curve of genus $\ell/2$ meeting the rest
of $X$ in a single Weierstrass point; we say that the cover has a {\em pointed tail} of genus $\ell/2$. Clearly, there is an isomorphism 
\[
\Exc(g) \simeq \H_{\ell+1}[k, \ell].
\]

Using functorial interpretation of morphisms $f$ and $g$ given above, a routine computation with explicit test families gives the following result.
\begin{lemma}\label{L:discrepancy} We have
\begin{multline*}
K_{\H_{n}[k,\ell]}+(\alpha+1/2)\delta_\irr+\delta_{\red}+(2\alpha+2\beta-1)\delta_W \\ -f^*(K_{\H_{n}[k+1, \ell]}+(\alpha+1/2)\delta_\irr+\delta_{\red}+(2\alpha+2\beta-1)\delta_W) =(1-(k+2)\alpha)\Exc(f)
\end{multline*}
and 
\begin{multline*}
K_{\H_{n}[k,\ell]}+(\alpha+1/2)\delta_\irr+\delta_{\red}+(2\alpha+2\beta-1)\delta_W \\ -g^*(K_{\H_{n}[k, \ell+1]}+(\alpha+1/2)\delta_\irr+\delta_{\red}+
(2\alpha+2\beta-1)\delta_W)=(1-(\ell+1)\alpha-\beta)\Exc(g), 
\end{multline*}
where $\Exc(f)$ and $\Exc(g)$ are effective (multiples of) 
divisors contracted by $f$ and $g$.
\end{lemma}

\section{Deformations of $D$ singularities}\label{S:D-singularities}

In this section, we explain how to view pointed quasi-admissible covers of Definition \ref{D:stack-H1}
as curves with $D$ singularities. More precisely, 
we establish the equivalence between deformations of an $A_{n-1}$ singularity with a 
section and a $D_n$ singularity. 


To begin, let $X=\spec \mathbb{K}[[x,y]]/(y^2-x^{n})$ be the $A_{n-1}$ singularity 
and $s=(0,0)\in X$ be a section. Consider the local deformation functor $\Def(X,s)$ which sends an Artinian $\mathbb{K}$\nb-algebra $A$ with the residue field 
$\mathbb{K}$ to the set of isomorphism classes of Cartesian diagrams 
\begin{align*}
\xymatrix{
X \ar[d] \ar[r] & \X \ar[d]\\
\spec \mathbb{K} \ar@/_/[u]_s \ar[r] & \spec A \ar@/_/[u]_\Sigma
}
\end{align*}
We have the following result:
\begin{prop}\label{P:A=D} 
The functor $\Def(X,s)$ is naturally isomorphic to $\Def(D_{n})$,
the local deformation functor of the $D_{n}$ singularity 
$\spec \mathbb{K}[[x,u]]/(x(u^2-x^{n-2}))$.
\end{prop}
\begin{proof}
To prove the statement, it suffices to establish an isomorphism between versal 
deformation spaces of $\Def(X,s)$ and $\Def(D_{n})$. To begin, 
a miniversal deformation space of an isolated planar singularity 
$\mathbb{K}[x,y]/(f(x,y))$ with a section $(x,y)=(0,0)$ can be taken to be the $\mathbb{K}$-vector space 
\[
\m/(f, \m\cdot(\partial f/\partial x, \partial f/\partial y)),
\]
where $\m=(x,y)\subset \mathbb{K}[x,y]/(f(x,y))$ is the maximal ideal 
(see, e.g., 
\cite[Section 2]{khirsch-martin} or \cite[Lemma 2.1]{mond-straten}).
 Moreover, if monomials $x^iy^j$
 form a basis of the 
said $\mathbb{K}$-vector space, then 
\[
\spec \mathbb{K}[x,y,\{t_{ij}\}]/(f(x,y)-\sum t_{ij}x^iy^j) \ra \spec \mathbb{K}[\{t_{ij}\}]
\] 
is the miniversal deformation with the universal section $\Sigma: \{x=y=0\}$. 
Applying this to $f(x,y)=y^2-x^n$, we obtain that 
$T:=\mathbb{K}[b,a_{0},\dots, a_{n-2}]$ is the base of the miniversal deformation of $\Def(X,s)$ 
and 
\[
\X:=\{y^2-by-x^{n}-a_{n-2}x^{n-1}-\dots-a_0x=0\}\subset \AA^2_{x,y}\times T
\] 
is the miniversal family. 

Next, let $\Sigma:\{x=0, y=0\}$ be the universal section, and $\Sigma':\{x=0, y-b=0\}$ be the conjugate section.
Since $\Sigma'$ is not a Cartier divisor on the total family $\X$, we can blow-up $\Sigma'$ to obtain
a new family
\[
\Y:=\Bl_{\Sigma'} \X
\] of plane curves over $T$. By construction, we can regard $\Y$ as a subvariety of $\PP^1_{[U:V]}\times \AA^2_{x,y} \times T$. We proceed to describe it by explicit 
equations: In the locus where 
$V\neq 0$, we have $y-b=xU/V$ and the equation of $\Y$ in terms of $x$ and $u:=U/V$ is 
\begin{align}\label{E:versal-D}
xu^2+ub-(x^{n-1}+a_{n-2}x^{n-2}+\dots+a_0)=0.
\end{align}
By the discussion in Section \ref{S:curves-definitions}, 
Equation \eqref{E:versal-D} defines precisely the miniversal deformation 
of the $D_n$ singularity $x(u^2-x^{n-2})=0$. 
In particular, the central fiber of $\Y \ra T$ has a unique singularity of type $D_{n}$, 
and the family $\Y \ra T$ is its versal deformation. This finishes 
the proof of the proposition.

We note that, more generally,
the blow-up $\Y \ra \X$ replaces a fiber of $\X\ra T$ in which 
the section $\Sigma$ coincides with an 
$A_{k-1}$ singularity by a fiber of $\Y\ra T$ with a $D_k$ singularity. This effect on
fibers is illustrated in Figure \ref{F:replacement}.
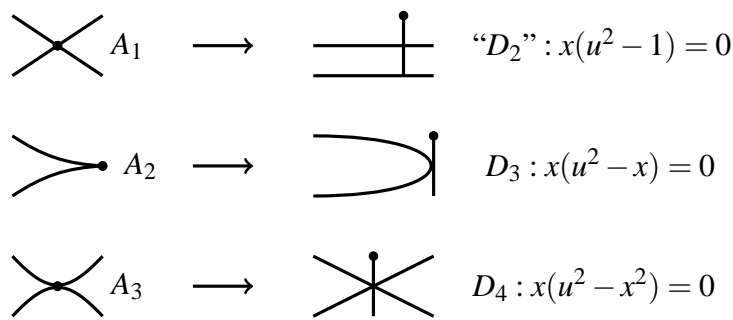
\begin{figure}[htb]
\begin{centering}
\begin{tikzpicture}[scale=0.8]
	\node  (0) at (-6.00, 4.00) {};
	\node  (1) at (-4.50, 4.00) {};
	\node  (2) at (0.50, 4.00) {};
	\node  (3) at (-5.25, 3.50) [label=right:$\quad A_1$] {};
	\draw[fill] (3) circle (2pt);
	\node  (d2) at (0.75, 3.50) [label=right:$\ensuremath{\quad ``D_2": x(u^2-1)=0}$] {};
	\draw[fill] (3) circle (2pt);
	\node  (4) at (-3.00, 3.50) {};
	\node  (5) at (-2.00, 3.50) {};
	\node  (6) at (-1.00, 3.50) {};
	\node  (7) at (1.00, 3.50) {};
	\node  (8) at (-6.00, 3.00) {};
	\node  (9) at (-4.50, 3.00) {};
	\node  (10) at (-1.00, 3.00) {};
	\node  (11) at (0.50, 3.00) {};
	\node  (12) at (1.00, 3.00) {};
	\node  (13) at (-6.00, 2.00) {};
	\node  (14) at (-1.00, 2.00) {};
	\node  (15) at (1.00, 2.00) {};
	\node  (16) at (-4.50, 1.50) [label=right:$A_2$] {};
	\draw[fill] (16) circle (2pt);
	\node  (d3) at (1.50, 1.50) [label=right:$\ensuremath{D_3: x(u^2-x)=0}$] {};
	\draw[fill] (16) circle (2pt);
	\node  (17) at (-3.00, 1.50) {};
	\node  (18) at (-2.00, 1.50) {};
	\node  (19) at (-2.00, 1.50) {};
	\node  (20) at (-2.00, 1.50) {};
	\node  (e1) at (0.50, 4.00) {};
	\draw[fill] (e1) circle (2pt);
	\node  (e2) at (1.00, 2.00) {};
	\draw[fill] (e2) circle (2pt);
	\node  (e3) at (0, 0) {};
	\draw[fill] (e3) circle (2pt);
	\node  (21) at (-6.00, 1.00) {};
	\node  (22) at (-1.00, 1.00) {};
	\node  (23) at (1.00, 1.00) {};
	\node  (24) at (-6.00, 0.00) {};
	\node  (25) at (-4.50, 0.00) {};
	\node  (26) at (-1.00, 0.00) {};
	\node  (27) at (0.00, 0.00) {};
	\node  (28) at (1.00, 0.00) {};
	\node  (29) at (-5.25, -0.50) [label=right:$\quad A_3$] {};
	\draw[fill] (29) circle (2pt);
	\node  (d4) at (0.75, -0.50) [label=right:$\ensuremath{\quad D_4: x(u^2-x^2)=0}$] {};
	\draw[fill] (29) circle (2pt);
	\node  (30) at (-3.00, -0.50) {};
	\node  (31) at (-2.00, -0.50) {};
	\node  (32) at (-6.00, -1.00) {};
	\node  (33) at (-4.50, -1.00) {};
	\node  (34) at (-1.00, -1.00) {};
	\node  (35) at (0.00, -1.00) {};
	\node  (36) at (1.00, -1.00) {};
	\draw  [very thick] (8.center) to (1.center);
	\draw [very thick, bend left=47, looseness=1.50] (32.center) to (33.center);
	\draw  [very thick] (6.center) to (7.center);
	\draw [very thick] (34.center) to (28.center);
	\draw [very thick, bend right=47, looseness=1.50] (24.center) to (25.center);
	\draw  [very thick, ->]  (30.center) to (31.center);
	\draw  [very thick] (2.center) to (11.center);
	\draw [very thick, bend right=15] (13.center) to (16.center);
	\draw  [very thick]  (10.center) to (12.center);
	\draw  [very thick]  (35.center) to (27.center);
	\draw  [very thick, ->]  (17.center) to (19.center);
	\draw  [very thick]  (15.center) to (23.center);
	\draw [very thick]  (0.center) to (9.center);
	\draw [very thick, bend right=90, looseness=6.7] (22.center) to (14.center);
	\draw   [very thick] (26.center) to (36.center);
	\draw [very thick, bend right=15] (16.center) to (21.center);
	\draw [very thick, ->] (4.center) to (5.center);
\end{tikzpicture}
\end{centering}
\caption{Replacing an $A_{k-1}$ singularity with a section by a 
$D_k$ singularity.}\label{F:replacement}
\end{figure}
\end{proof}

\section{$(A,D)$-stable reduction} 
\label{S:final}

Given a proper integral curve $C$ of arithmetic genus $g \gg 0$ with a single isolated smoothable 
singularity $p$, one has a rational {\em moduli map} 
$
j\co \Def(C) \dashrightarrow \Mg{g}
$. 
The problem of resolving the indeterminacy of $j$ plays an important 
role in the study of alternate compactifications of $\mathcal{M}_{g}$ (cf. Section
\ref{S:intro}). 
Recently, Casalaina-Martin and Laza have described an explicit alteration 
sufficient to regularize the 
moduli map in the 
case when $\hat{\O}_{C,p}$ is an ADE singularity \cite[Main Theorem]{radu-yano}. We briefly describe 
their approach: 

First, one passes to the Weyl cover of $\Def(C)$ (see \cite[Section 2]{radu-yano} and references therein) 
-- a finite base extension after which the discriminant divisor $\Delta_{C}\subset \Def(C)$ becomes an 
arrangement of hyperplanes.
After the wonderful blow-up of \cite{wonderful} the discriminant becomes a simple normal crossing divisor.  An application of the extension theorem of de Jong-Oort 
\cite[Theorem 5.1]{dejong-oort} (see also \cite[Theorem 1.2]{sabin}) 
now gives a morphism to the coarse moduli space $\M_g$. Finally, one verifies that the family of stable curves away from the discriminant extends to a family over the generic point of every irreducible component of the discriminant, except for the components corresponding to $A_{2k}$
singularities.  A further finite base change is required to obtain the extension of the morphism to the moduli stack $\Mg{g}$. We refer to \cite{radu-yano}
 for more details.

An application of our Main Theorem \ref{main2} is
an iterative functorial resolution of the indeterminacy of 
the moduli map $j\co \Def(C) \dashrightarrow \Mg{g}$
in the case when $C$ has only $A$ and $D$ singularities. 
\begin{theorem}[$(A_{k},D_{\ell})$-stable reduction]
\label{T:resolution} 
Suppose $\C\ra \Def(D_{n})$ is the 
miniversal deformation ($n\geq 4$). Then for integers $k$ and $\ell$ 
such that 
$\ell\leq \min\{k+1,n\}$, there  
is a representable by proper Deligne-Mumford stacks 
morphism $f_{k,\ell}\co T \ra \Def(D_{n})$, which is an isomorphism away
from $\Delta\subset \Def(D_{n})$, such that 
$\C\vert_{\Def(D_{n})\smallsetminus \Delta}$ extends to the family of curves
with at worst $A_{k}$ and $D_{\ell}$ singularities over $T$. 
\end{theorem}
\begin{proof}
To begin, 
consider the moduli stack 
$\H_{n+1}[n,n]$ with the universal quasi-admissible cover 
$\varphi\co \X \ra \Y$. 
We restrict to a neighborhood of the distinguished point in
$\H_{n+1}[n,n]$ corresponding to the unique 
quasi-admissible cover with a $D_{n}$ singularity (cf. Proposition \ref{P:A=D}). 
By Proposition \ref{P:versality2} , there is an \'{e}tale
neighborhood $U$ of this point that is isomorphic to $\Def(D_n)$. 
Moreover, the restriction of $\X\ra \H_{n+1}[n,n]$ to $U$ is the miniversal deformation of $D_n$, and if $\delta:=\delta_{\irr}
\cup \delta_{\red}\cup \delta_{W}\subset \H_{n+1}[n,n]$, then 
$U\cap \delta=\Delta\subset \Def(D_{n})$.

Consider now the reduction morphism 
$f\co\H_{n+1}[k,\ell] \ra \H_{n+1}[n,n]$. 
By construction, $f$ is an isomorphism over 
$\H_{n+1}[n,n] \smallsetminus \delta$ and the universal cover over 
$\H_{n+1}[k,\ell]$
agrees with $\X$ over $\H_{n+1}[n,n] \smallsetminus \delta$. 
By definition, 
the universal cover over $\H_{n+1}[k,\ell]$ is a family 
of 
quasi-admissible covers with at worst $A_k$ and $D_{\ell}$ 
singularities. Passing to the relative coarse 
moduli space over $\H_{n+1}[k,\ell]$
(this is possible because the characteristic is $0$), 
we obtain a family of curves with at worst $A_k$ and $D_{\ell}$ singularities
extending the 
miniversal family over $\Def(D_{n})\smallsetminus \Delta$. 
It follows that 
\[
f_{k,\ell}:=f\vert_{f^{-1}(U)} \co f^{-1}(U)\ra U\simeq \Def(D_{n}),
\]
is a requisite morphism. 
\end{proof}
We note that in practice one desires an explicit blow-up procedure that, for an arbitrary family of 
curves, allows one to replace $A_k$ singularities in the fibers 
by hyperelliptic tails with at worst $A_{k-1}$ singularities, 
and to replace $D_\ell$ singularities in the fibers by hyperelliptic tails with at worst $A_{\ell-1}$ and $D_{\ell-1}$ 
 singularities.
Proposition \ref{P:A=D} (and explicit blow-ups presented in its proof) 
reduces the problem to that for $A$ singularities only and 
 the following proposition (essentially generalizing \cite[Section 5.1]{severi} to the case of all $k$)
 is an illustration of Theorem \ref{T:resolution} for the $A$ case. 
\begin{prop}\label{P:explicit-A-stable-reduction}
For a miniversal family $\C\ra T$ of an $A_{k}$-singularity ($y^{2}=x^{k+1}$)
there is an alteration $f\co T'\ra T$ and a weighted 
blow-up $\C' \ra \C\times_{T}T'$ of the $A_{k}$-locus in the fibers such that
$\C'\ra T'$ is a flat family of curves with at worst $A_{k-1}$ singularities and
such that $\C'\vert_{f^{-1}(0)}\simeq \Y_1\cup \Y_2$ is a union of two irreducible components such that
$\Y_1\ra f^{-1}(0)$ is an isotrivial family of normalizations of the central fiber $\C_0$ 
and $\Y_2\ra f^{-1}(0)$ is 
a 
family of curves in $\H_k[k-1]$.
\end{prop}

\begin{proof}
We can assume that $T\simeq \spec \mathbb{K}[a_{0}, a_{1}, \dots, a_{k-1}]$
and the miniversal family $\C$ is given by the equation 
\[
y^{2}=x^{k+1}+a_{k-1}x^{k-1}+\cdots+a_{1}x+a_{0}.
\] 
We begin with a finite base change 
$a_{i}=b_{i}^{k+1-i}$. Set $T'$ to be the blow-up of $\spec \mathbb{K}[b_{0}, \dots, b_{k-1}]$ along the ideal $(b_{0},\dots, b_{k-1})$ and denote 
the resulting morphism $T'\ra T$ by $f$. By construction, $f$ is 
a composition of a faithfully flat finite and a proper birational morphisms. 
Consider the customary affine cover $T'=\bigcup_{j=0}^{k-1} U_{j}$ where
$U_{j}=\spec \mathbb{K}[u, c_{0},\dots, \hat{c_{j}}, \dots, c_{k-1}]$. Then the morphism 
$U_{j}\ra T$ is given by
\[
b_i\mapsto uc_i \text{  \  for \  $i\neq j$ and}  \ b_j\mapsto u.
\]
Note that the exceptional divisor $E:=f^{-1}(0)$ of $f$ is defined by the equation
$u=0$ on $U_{j}$.
\par 
We claim that performing a weighted blow-up of $\C\times_{T}T'$ 
with $\text{weight}(x,y,u)=(2,k+1,2)$, or, 
to put it differently, taking
$\C':=\Bl_{\J}(\C\times_{T}T')$ where $\J=((\I_{E},x)^{(k+1)/2},y)$ if $k$ is odd, and 
$\J=((\I_{E},x)^{k+1},y^2, y(\I_{E},x)^{k/2+1})$ if $k$ is even, gives us a requisite family. 
To begin, 
denote the exceptional divisor of $\C'=\Bl_{\J}(\C\times_{T}T')\ra \C\times_{T}T'$ by $\Y_2$
and the strict transform of $\C_0\times f^{-1}(0)$ by $\Y_1$.
To check the assertion,  
we work over the affine patch 
$U_{j}=\spec \mathbb{K}[u, c_{0},\dots, \hat{c_{j}}, \dots, c_{k-1}]$ over which the equation 
of $\C\times_{T} T'$ is 
\[
y^{2}=x^{k+1}+c^2_{k-1}u^{2}x^{k-1}+\cdots+u^{j}x^{k+1-j}+\cdots+c^{k+1}_{0}u^{k+1}.
\]
 It is easy to see that $\Y_1\ra E\cap U_j$ is a trivial family whose fiber is the normalization of 
 the central fiber $\C_0$. 
Further, $\Y_2$ is a family of divisors in
the weighted projective space $\PP(2,2,k+1)$ given by the quasi-homogeneous (in variables $x,u,y$)
equation 
\begin{multline*}
\{y^{2}=x^{k+1}+c_{k-1}^{2}u^{2}x^{n-1}+\cdots+u^{j}x^{k+1-j}+\cdots+c_{0}^{k+1}u^{k+1} \}
\\ \subset \PP(2,2,k+1)\times (E\cap U_{j}).
\end{multline*}
In particular, $\Y_2$ admits a $2:1$ morphism to $\PP(2,2)\times (E\cap U_j)\simeq \PP^1\times (E\cap U_j)$ given by 
$(x,u,y)\mapsto (x,u)$. Since the ramification divisor has degree $k+1$ when $k$ is odd and $k+2$ 
when $k$ is even and has no points of multiplicity $k+1$, we conclude that $\Y_2\ra (E\cap U_j)$ is 
a family of curves in $\H_{k}[k-1]$. 
Moreover, $\Y_2$ is attached to $\Y_1$ along the locus $(u=0)\cap \Y_2$ 
in $\PP(2,2,k+1)\times (E\cap U_j)$. Since the equation $y^2=x^{k+1}$ defines two points in $\PP(2,k+1)$
when $k$ is odd and a single point when $k$ is even, we conclude that, for every $t\in f^{-1}(0)$,
the curve $(\Y_2)_t$ is attached 
$(\Y_1)_t$ along a ramification point when $k$ is even and along two conjugate points if $k$ is odd.
\end{proof}


\begin{remark}[Stable reduction over higher-dimensional bases] 
Note that given an irreducible scheme $T$ of any dimension 
and a stable curve over the generic 
point of $T$, there is always an alteration\footnote{An alteration is 
a composition of a proper birational morphism with a finite morphism
\cite{dejong-alterations}.} 
of $T$ after which (the pullback of) 
the stable curve over the generic fiber extends over the whole base.
Indeed, since $\Mg{g}$ is a proper stack over $\ZZ$ with a 
finite diagonal, by \cite[Theorem 2.7]{quotient-stacks} there exists a finite surjective 
morphism $V\ra \Mg{g}$, 
where $V$ a scheme. The rational
map $T\dashrightarrow \Mg{g}$ lifts to a rational map 
$T' \dashrightarrow V$, where $T'$ maps finitely to $T$. 
Since $V$ is proper, there exists 
a proper birational morphism $T''\ra T'$ after which this rational map 
extends. The composition $T''\ra T' \ra T$ is a requisite alteration.
\end{remark}


We now explain where the necessity in working with Deligne-Mumford stacks 
originates. This is accomplished by the following proposition.
\begin{prop}\label{P:nonexistence} There exists a curve $C$ with a unique isolated 
$A_{2k}$ singularity 
such that the indeterminacy of the moduli map 
$\Def(C)\dashrightarrow \Mg{p_{a}(C)}$ cannot be resolved 
by any proper birational modification of $\Def(C)$ which restricts to an isomorphism 
over $\Def(C)\smallsetminus \Delta_{C}$. Furthermore, the indeterminacy 
cannot be resolved by any proper modification of $\Def(C)$ which is isomorphic
to the Weyl cover of $\Def(C)$ over $\Def(C)\smallsetminus \Delta_{C}$.
\end{prop}

\begin{proof}
To prove the first part it suffices to exhibit $C$ and a smoothing $f\co \C\ra (T,0)$ 
of $C$ such that the stable curve 
$\C\times_{T}(T\smallsetminus 0) \ra T\smallsetminus 0$ does not 
extend to a stable curve over $T$. 

To begin, consider a 
family of stable $(4k+2)$-pointed rational curves $\Y \ra (T,0)$ (over a spectrum of a DVR)
with the central fiber 
$Y_0=E_1\cup E_2$ --
a nodal union of two rational curves each marked by exactly $2k+1$ sections. 
Assume that the total space of $\Y$ is smooth.
Denote by $\Sigma$ the union of all $4k+2$ sections. 
Next, let $T' \ra T$ be a finite base extension of degree $2$, ramified over $0$.
The fiber product $\Y'=\Y\times_{T} T'$ has a (surface) singularity of type $A_1$ 
lying over the node of $Y_0$. Make an ordinary blow-up with the center at this 
singularity, and denote by $F$ the exceptional divisor. 
By, e.g., Lemma \ref{L:even-line-bundle}, the divisor $\Sigma+F$ is divisible 
by $2$ in the Picard group of the blown-up surface.
We can now construct a cyclic $2$-cover branched over $\Sigma+F$.
The resulting cover is smooth and contains a $(-1)$-curve -- the preimage of $F$.
Consider the Stein factorization of the morphism from the said cyclic cover to $\Y'$.
The resulting Stein morphism blows-down the $(-1)$-curve, and the resulting surface $\X'$ 
admits a finite, degree $2$ morphism to $Y'$. 

Note that, by construction, the central fiber of $\X'$ is a nodal union 
$X_1\cup X_2$ of two hyperelliptic genus $k$ curves. 
The line bundle $\omega_{\X'/T'}((2k-1)X_1)$ is relatively base-point-free: 
Clearly, it has no base points away from $X_{1}$, and by considering the divisor
$(4k-2)\tau$, where $\tau$ is a Weierstrass section of $\X'/T'$ disjoint from $X_{1}$,
we conclude that it has no base points along $X_{1}$. 
Since $\omega_{\X'/T'}((2k-1)X_1)\vert_{X_{1}}\simeq \O_{X_{1}}$,
it follows that $\omega_{\X'/T'}((2k-1)X_1)$ defines a $T'$-morphism $\X'\ra \X''$
contracting $X_1$ to a point. 
A simple application the theorem on formal functions shows that 
the image of $X_{1}$ is an $A_{2k}$ singularity of 
the central fiber $X''_0$.

Finally, we take $C:=X''_{0}$. 
By the construction,
we have that $\X''_{\vert T'\smallsetminus 0}$ is a base extension of a 
smooth family of genus $2k$ curves  
over $T\smallsetminus 0$
and so the morphism $T'\smallsetminus 0\ra \Mg{2k}$ factors through $T \smallsetminus 0$.  
However, since $\X''$
is smooth, $\X''\ra T'$ cannot be a base extension of a stable family of genus $2k$ curves over
 $T$. It follows that the morphism $T \smallsetminus 0 \ra \Mg{2k}$ does not extend to $T$. This establishes the first part of the proposition. 
 
 The second part follows analogously due to the fact that, by the above construction, 
 the morphism from $T$ to $\Def(C)$ factors through the Weyl cover of $\Def(C)$.

\end{proof}

\begin{remark}
We remark that Proposition \ref{P:nonexistence} strengthens 
the non-existence part of \cite[Theorem 6.1]{radu-yano}, 
which is proven there using monodromy considerations. 
\end{remark}

\bibliography{ADpaperBIB}
\bibliographystyle{alpha}
\end{document}